\documentclass[11pt]{article}
\usepackage[english]{babel}
\usepackage{xcolor}
\usepackage{amsmath}
\usepackage{amsthm}
\usepackage{amsfonts}
\usepackage{amssymb}
\usepackage{euscript}
\usepackage[cp1251]{inputenc}
\usepackage[pdftex]{graphicx}
\usepackage[pdfpagelabels,bookmarksdepth=3]{hyperref}
\hypersetup{colorlinks=true,citecolor=black,urlcolor=black,linkcolor=black}

\newtheorem{theorem}{Theorem}[section]
\newtheorem{lemma}[theorem]{Lemma}

\theoremstyle{definition}

\theoremstyle{remark}

\numberwithin{equation}{section}

\begin{document}
\begin{flushleft}
\textbf{\Large{Monotone Lagrangian submanifolds of $\mathbb{C}^n$ and toric topology}}
\end{flushleft}
\begin{flushleft}
\textbf{Vardan Oganesyan}\\
\end{flushleft}

\textbf{Abstract.} Mironov, Panov and Kotelskiy studied  Hamiltonian-minimal Lagrangians inside $\mathbb{C}^n$. They associated a closed embedded Lagrangian $L$ to each Delzant polytope $P$. In this paper we develop their ideas and prove that $L$ is monotone if and only if the polytope $P$ is Fano.

In some examples, we further compute the minimal Maslov numbers. Namely, let $\mathcal{N}\to T^k$ be some fibration over the $k$-dimensional torus with a fiber equal to either $S^k \times S^l$, or $S^k \times S^l \times S^m$, or $\#_5(S^{2p-1} \times S^{n-2p-2})$. We construct monotone Lagrangian embeddings $\mathcal{N} \subset \mathbb{C}^n$ with different minimal Maslov number, and therefore distinct up to Lagrangian isotopy. Moreover, we show that some of our embeddings are smoothly isotopic but not Lagrangian isotopic.
\tableofcontents

\newpage
\section{Introduction}

The simplest example of a symplectic manifold is $\mathbb{C}^n$ with the standard symplectic form. There is a natural question:
\begin{center}
\emph{What can be said about the topology of a closed \\
Lagrangian submanifold $L \subset \mathbb{C}^n$? }
\end{center}
Known restriction results include:

\noindent - Lagrangian $L\subset \mathbb{C}^n$ cannot be simply connected (Gromov~\cite{Gromov}). For example, a sphere cannot be embedded as a Lagrangian submanifold into $\mathbb{C}^n$.

\noindent
- If $n=2$, Lagrangian $L\subset \mathbb{C}^2$ must be diffeomorphic to either the torus $T^2$, or a nonorientable surface with Euler characteristic divisible by $4$ (Nemirovsky~\cite{Nemirovsky}, Shevchishin~\cite{Shev}, Givental~\cite{Givental}).

\noindent
- If $n=3$ and Lagrangian $L \subset \mathbb{C}^3$ is a closed orientable prime\footnote{A $3$-manifold $L$ is called \emph{prime} if any decomposition $L = L_1 \# L_2$ implies that either $L_1$ or $L_2$ is diffeomorphic to $S^3$.} 3-manifold, then $L$ must be diffeomorphic to $S^1\times S_g$, a product of  a circle  and  a closed orientable surface of genus $g$ (Fukaya~\cite{Fukaya}).

\noindent
- If $n > 3$, much less is known. There are many examples of Lagrangian submanifolds constructed from Lagrangian immersions and replacing neighborhoods of the double points by 1-handles (Polterovich~\cite{Polterovich2}).

\medskip

A subclass of \emph{monotone} Lagrangian submanifolds is of special interest due to their prominent role in Floer theory. The basic question is:
\begin{center}
\emph{Given a closed Lagrangian submanifolds $L \subset \mathbb{C}^n$, \\does $L$ also admit a monotone Lagrangian embedding into $\mathbb{C}^n$?}
\end{center}
The monotonicity condition turns out to be very restrictive. In~\cite{Evans,Damian2}, Damian, Evans, Fukaya and Kedra show that if $L$ is monotone closed orientable Lagrangian submanifold of $\mathbb{C}^3$, then $L$ must be diffeomorphic to $S^1 \times S_g$ (note that $L$ is not assumed to be prime), where $S_g$ is a closed orientable surface of genus g.

The monotonicity condition also gives some restrictions on the minimal Maslov number, which we denote by $N_L$. In \cite{Oh2}, Oh constructed a spectral sequence, which starts with Morse cohomology and converges to Floer cohomology. Differentials of this spectral sequence depend on the minimal Maslov number. Using this spectral sequence Oh proved that if $L$ is a closed monotone Lagrangian submanifold of $\mathbb{C}^n$, then $1 \leqslant N_L \leqslant n$.

The following Lagrangian isotopy question is a starting point of this paper:
\begin{center}
\emph{Given two Lagrangian embeddings $i,i':L \hookrightarrow \mathbb{C}^n$, \\ are $i$ and $i'$ Lagrangian isotopic (Hamiltonian isotopic)?}
\end{center}
In papers \cite{Chek1,Chek2}, Chekanov and Schlenk found examples of Lagrangian tori in $\mathbb{C}^n$ that are not Hamiltonian isotopic to the standard torus $(S^1)^n$.  Infinite families of Hamiltonian non-isotopic monotone Lagrangian tori were constructed by Vianna in $\mathbb C P^2$~\cite{Vianna1,Vianna2}, and by Auroux in $\mathbb{C}^3$~\cite{Auroux}. Many further interesting examples are discovered by Mikhalkin~\cite{Mik}.

In papers \cite{MironovCn, Mirpan, kotel} Mironov, Panov  and Kotelskiy studied Hamiltonian-minimal and minimal Lagrangian submanifolds of toric manifolds. In particular, they associated a closed Hamiltonian-minimal Lagrangian submanifold  $L \subset \mathbb{C}^n$ to each Delzant polytope. The Lagrangian $L$ is diffeomorphic to the total space of  fiber bundle over $T^k$, where the fiber is the so-called real moment-angle manifold associated to $P$. In some cases real moment angle manifolds are diffeomorphic to connected sums of sphere products, but in general real moment-angle manifolds define a rich family of smooth manifolds and their topology is far from being completely understood (see \cite{Lopez, real1, real2}).

In this paper, we further study the family of Lagrangians associated to Delzant polytopes, and construct new examples of \emph{monotone} Lagrangian sumbanifolds of $\mathbb{C}^n$. Moreover, we compute the Maslov class and prove that submanifolds are not Lagrangian isotopic. As an application, we discover a Lagrangian rigidity phenomenon: by appealing to the Haefliger-Hirsch classification of smooth isotopy classes~\cite{Hirsch}, we provide a positive answer to the following question:
\begin{center}
\emph{Are there two Lagrangian embeddings $i,i':L \hookrightarrow \mathbb{C}^n$ that are smoothly isotopic but not Lagrangian isotopic?}
\end{center}

We now proceed to precise formulations of our results. Let $P$ be a Delzant polytope in $\mathbb{R}^k$ with $n$ facets, where $n>k$. As it is mentioned above we associate an embedded Lagrangian $L \subset \mathbb{C}^n$ to each Delzant polytope $P$. Our first result is:

\begin{theorem}\label{Fanos}
\emph{Let $P$ be a Delzant and irredundant polytope and $L \subset \mathbb{C}^n$ be the corresponding Lagrangian submanifold. Then $L$ is monotone if and only if the polytope $P$ is Fano.}
\end{theorem}

There are many Delzant polytopes (infinitely many even in dimension $2$). Also, there are many Delzant polytopes that are Fano. As a result we obtain a large family of monotone Lagrangian submanifolds. Moreover, many polytopes provide the same diffeomorphism type of $L$. As a simplest nontrivial example, let us consider the product of two simplices $\Delta^{p-1} \times \Delta^{n-p-1} \subset \mathbb{R}^{n-2}$ defined by the inequalities
\begin{equation*}
\begin{gathered}
\left\{
 \begin{array}{l}
x_i + 1 \geq 0 \quad i=1,...,p-1 \\
-x_1 - ... - x_{p-1} +1 \geq 0 \\
x_i + 1 \geq 0 \quad i=p,...,n-1 \\
-x_1 - ... - x_{k} - x_{p} - ... - x_{n-1} + 1 \geq 0  \\
 \end{array}
\right.
\end{gathered}
\end{equation*}
where $n-p+k > p, \; \; k < p-1$.
Denote by $\gcd(a,b)$ the greatest common divisor of $a$ and $b$.

\begin{theorem}\label{ex1}
Let $P_k$ be the polytope defined above, and $L_k \subset \mathbb{C}^n$ be the corresponding Lagrangian submanifold. As an abstract manifold, $L_k$ fibers over $2-$torus $T^2$  with fiber $S^{p-1} \times S^{n-p-1}$. We have:
\begin{itemize}
  \item[-] Lagrangians $L_k$ are monotone;
  \item[-] The minimal Maslov numbers are given by $N_{L_k} = \gcd(p, n-p+k)$;
  \item[-] If $k,p,n$ are even numbers, then the fibration is trivial and $L_k$ is diffeomorphic to $S^{p-1} \times S^{n-p-1} \times T^2$;
  \item[-] As a consequence, we see that if $n, p, k$ are even, then the diffeomorphism type of $L_k$ is independent of $k$, but the Maslov class depends on $k$. Therefore, we get examples of monotone Lagrangians distinct up to Lagrangian isotopy;
  \item[-] Some of the Lagrangians $L_{2k}$ are smoothly isotopic but not Lagrangian isotopic.
  \item[-] The fibration is orientable if and only if numbers $p$ and $n-p+k$ are even.
\end{itemize}
\end{theorem}

In a similar way we can consider product of three simplices and prove the following theorem:

\begin{theorem}\label{ex2}
Let $n,p,k,q,l$ be arbitrary even positive integers such that
\begin{equation*}
\begin{gathered}
q < l < k < p < n, \quad k-l-q  < 0, \quad n-p+l < p-k+q.
\end{gathered}
\end{equation*}
There exists an embedding of
\begin{equation*}
\mathcal{N} = S^{n-p+k-q-1} \times S^{p-k-1} \times S^{q-1} \times T^3
\end{equation*}
into $\mathbb{C}^n$ with minimal Maslov number
\begin{equation*}
N_{\mathcal{N}} = gcd(n-p+l, l+q-k, p-k+q),
\end{equation*}
We see that the  diffeomorphism type of $\mathcal{N}$ is independent of $l$, unlike the Maslov class. Thus we obtain embeddings of $\mathcal{N}$ distinct up to Lagrangian isotopy. Some of these embeddings are smoothly isotopic but not Lagrangian isotopic.

In general, if $n,p,k,q,l$ are not even, then $\mathcal{N}$ fibers over the $3-$torus $T^3$ with fiber $S^{n-p+k-q-1} \times S^{p-k-1} \times S^{q-1}$. The  fibration is orientable if and only if numbers $n-p+l$, $k-l+q$, $p-k+q$ are even.
\end{theorem}

As we mentioned above, if $L \subset \mathbb{C}^3$ is a closed orientable monotone Lagrangian, then $L$ is diffeomorphic to $S_g \times S^1$. Let us think about $S_g \times S^1$ as a trivial fiber bundle over $S^1$ with fiber $S_g = \#_g(S^1 \times S^1)$. The following theorems can be considered as a generalization of this example.

\begin{theorem}\label{th3}
Let $P \subset \mathbb{R}^2$ be a Delzant and Fano pentagon and $L \subset \mathbb{C}^5$ be the corresponding Lagrangian. Then $L$ is diffeomorphic to a fiber bundle over $T^3$ with fiber an oriented surface of genus $5$. Moreover, $L$ is embedded monotone submanifold of $\mathbb{C}^5$ . The fiber bundle is nonorientable and the minimal Maslov number of $L$ is equal to $1$.
\end{theorem}

\begin{theorem}\label{th4}
Let $p,q$ be arbitrary even positive numbers such that $q < p-1$. There exists an embedding of $\mathcal{N} = \#_5(S^{2p-1} \times S^{3p-2}) \times T^3$ into $\mathbb{C}^{5p}$ with the minimal Maslov number $gcd(p, q)$. As in the previous theorems we see that the diffeomorphism type of the Lagrangian is independent of $q$, unlike the Maslov class. As such, we obtain embeddings distinct up to Lagrangian isotopy, and again, some of these embeddings are smoothly isotopic but not Lagrangian isotopic.

\end{theorem}

\begin{theorem}\label{th5}
Let $P \subset \mathbb{R}^2$ be a Delzant and Fano $6$-gon and $L \subset \mathbb{C}^6$ be the corresponding Lagrangian. Then $L$ is diffeomorphic to a fiber bundle over $T^4$ with fiber an oriented surface of genus $17$. Moreover, $L$ is embedded monotone submanifold of $\mathbb{C}^6$. The fiber bundle is nonorientable and the minimal Maslov number of $L$ is equal to $1$.
\end{theorem}

Our method allows to construct immersed monotone Lagrangian submanifolds. We need to consider Fano polytopes which are not Delzant. For instance,  we obtain the following result:

\begin{theorem}\label{th6}
Let $k$ be an arbitrary even integer greater than $3$. There exists a manifold $\mathcal{N}$, which can be immersed into $\mathbb{C}^5$ as a monotone Lagrangian with minimal Maslov number $k$. Manifold $\mathcal{N}$ is diffeomorphic to a fiber bundle over $T^3$ with fiber an oriented surface of genus $5$, where the fiber bundle is orientable. As a result, we get infinitely many monotone immersions  into $\mathbb{C}^5$ distinct up to Lagrangian isotopy.

\end{theorem}

In general, if a monotone Lagrangian $L \subset \mathbb{C}^n$ comes from a Delzant Fano polytope, $L$ is the total space of a fiber bundle over $T^{n-k}$, where the fiber is the real moment-angle manifold associated to $P$. The problem of existence of homeomorphic but not diffeomorphic real moment-angle manifolds is open. This problem leads to the following question:
\begin{center}
\emph{Does our method allow to construct homeomorphic \\but not diffeomorphic monotone embedded Lagrangians?}
\end{center}

\noindent \textbf{Acknowledgments.}
The author thanks Artem Kotelskiy, Mark McLean and Yuhan Sun for many helpful discussions.

This work was supported by the Russian Science Foundation under grant no.18-11-00316 and performed in L.D. Landau Institute for Theoretical Physics.

\section{Preliminary definitions and results}\label{definitions}

An immersion $\psi: L\rightarrow \mathbb{C}^n$ of an $n-$dimensional manifold $L$  is called Lagrangian if $\psi^{*}\omega$ = 0, where
\begin{equation*}
\omega = \frac{i}{2}\sum\limits_{j=1}^{n}dz^j \wedge d\overline{z}^j = \sum\limits_{j=1}^{n}dx^j \wedge dy^j, \quad z^j = x^j + iy^j.
\end{equation*}
We know that
\begin{equation*}
\omega = d\lambda = d(x^1dy^1 + ... +x^ndy^n).
\end{equation*}
Assume $\alpha \in H_1(L, \mathbb{Z})$. Then we define
\begin{equation*}
\lambda(\alpha) = \int\limits_{\alpha}\lambda.
\end{equation*}
\textbf{Definition}. The homomorphism $\lambda: H_1(L, \mathbb{Z}) \rightarrow \mathbb{R}$ is called the symplectic area class.
\\

Let $H$ be the mean curvature vector of $L \subset \mathbb{C}^n$, where $L$ is a Lagrangian submanifold. By $\omega_H$ denote the 1-form $\omega(H,\cdot)|_{TL}$. It is known that $d\omega_{H} = 0$ (see \cite{Oh}). We are not giving the general definition of the Maslov class. Instead we define the Maslov class in the following way:
\\
\\
\textbf{Definition}(see \cite{Kail} for details). The Maslov class $\mu$ is given by
\begin{equation}\label{meancurve}
\begin{gathered}
\mu: H_1(L, \mathbb{Z}) \rightarrow \mathbb{Z},\\
\mu(\alpha) =  \frac{1}{\pi}\int\limits_{\alpha}\omega_{H}.\\
\end{gathered}
\end{equation}

\textbf{Definition}. A Lagrangian $L \subset \mathbb{C}^n$ is called monotone if there exists a constant  $c > 0$ such that $\lambda(\alpha)=c\mu(\alpha)$ for all $\alpha \in \pi_1(L)$.
\\

It is known that there are no closed minimal submanifolds of $\mathbb{C}^n$. Instead of minimal submanifolds we can consider Hamiltonian-minimal (or H-minimal) submanifolds.
\\
\\
\textbf{Definition}(\cite{Oh}). A Lagrangian immersion $\psi : L \rightarrow \mathbb{C}^n$ is called \emph{H-minimal} if the variations of volume $\psi(L)$ along all Hamiltonian vector fields with compact supports vanish. In other words
\begin{equation*}
\frac{d}{dt}vol(\psi_{t}(L))|_{t=0} = 0,
\end{equation*}
where $\psi_{0}(L)=\psi(L)$, $\psi_t(L)$ is a deformation  of $\psi(L)$ along some Hamiltonian vector field.
\\
The notion of H-minimality was introduced by Oh \cite{Oh}.
\\
\\
\textbf{Definition}. Two embedded Lagrangian submanifolds $L_1,=\psi_1(L) \subset \mathbb{C}^n$ and $L_2=\psi_2(L) \subset \mathbb{C}^n$  are called isotopic if there exists a smooth map $h_t: L\times[0,1] \rightarrow \mathbb{C}^n$ such that $h_t$ is an embedding for any $t$ and $h_0(L)=L_1$, $h_1(L)=L_2$.
\\
\\
\textbf{Definition}. Two Lagrangian submanifolds $L_1=\psi_1(L) \subset \mathbb{C}^n$ and $L_2=\psi_2(L) \subset \mathbb{C}^n$  are called Lagrangian isotopic if there exists a smooth map $h_t: L\times[0,1] \rightarrow \mathbb{C}^n$ such that $h_t$ is a Lagrangian embedding for any $t$ and $h_0(L)=L_1$, $h_1(L)=L_2$.
\\

If Lagrangian submanofolds $L_1, L_2$ are Lagrangian isotopic, then they are smoothly isotopic.
\\
\\
\textbf{Definition.} Let $L \subset \mathbb{C}^n$ be an embedded (immersed) Lagrangian submanifold. A nonnegative generator $N_L$ of the subgroup $\mu(H_1(L, \mathbb{Z})) \subset \mathbb{Z}$ is called the minimal Maslov number, where $\mu$ is the Maslov class.
\\

Lagrangian isotopy preserves the number $N_L$.  This means that if $N_{L_1} \neq N_{L_2}$ for two embeddings $L_1=\psi_1(L), L_2=\psi_2(L)$, then $L_1$ is not Lagrangian isotopic to $L_2$.

In paper \cite{Hirsch} Haefliger and Hirsch classified smooth embeddings of compact n-manifolds into $\mathbb{C}^n$ up to smooth isotopy, where $n \geqslant 5$.

\begin{theorem}\label{Haef} (see \cite{Hirsch}) Let L be a closed, oriented, connected $n$-manifold and suppose that $\psi : L \rightarrow \mathbb{C}^n$ is an embedding, where $n \geqslant 5$. The isotopy classes of smooth embeddings are in bijection with the elements of
\begin{equation*}
\left\{
 \begin{array}{l}
H_1(L, \mathbb{Z}) \quad if \; \; \; n \;\; \; is \;\; \; odd \\
H_1(L, \mathbb{Z}_2) \quad if \; \; \; n \;\; \; is \;\; \; even
 \end{array}
\right.
\end{equation*}
\end{theorem}

\section{Hamiltonian-minimal Lagrangian submanifolds}\label{Ham}

Mironov in \cite{MironovCn} found a very interesting method for constructing H-minimal Lagrangian submanifolds of $\mathbb{C}^n$. It turns out that the topology of constructed submanifolds can be highly complicated. But methods of toric topology give us a technique to study our submanifolds. We discuss toric topology in Section $\ref{torictopology}$. Let us briefly explain Mironov's method.
\\

Let $\mathcal{R}$ be a $k$-dimensional submanifold of $\mathbb{R}^n$ defined by the system of equations
\begin{equation}\label{equation}
\begin{gathered}
\left\{
 \begin{array}{l}
\gamma_{1,i}u_1^2 + ... + \gamma_{n,i}u_n^2 = \delta_i,
 \end{array}
\right.
i=1,...,n-k,
\end{gathered}
\end{equation}
where $\delta_j \in \mathbb{R}$, $\gamma_{i,j} \in \mathbb{Z}$. Let us assume that the integer vectors
\begin{equation}\label{vectors}
\gamma_j=(\gamma_{j,1},...,\gamma_{j,(n-k)})^T \in \mathbb{Z}^{n-k}, \quad j=1,...,n
\end{equation}
form a lattice $\Lambda$ in $\mathbb{R}^{n-k}$ of maximum rank. Let $\Gamma$ be the matrix with columns $\gamma_j$, $j=1,...,n$. The dual lattice $\Lambda^{*}$ is defined by
\begin{equation*}
\Lambda^{*}=\{\lambda^{*} \in \mathbb{R}^{n-k}| (\lambda^{*},\lambda) \in \mathbb{Z}, \lambda \in \Lambda\},
\end{equation*}
where $(\lambda^{*},\lambda)$ is the Euclidian product on $\mathbb{R}^{n-k}$. Let $D_{\Gamma}$ be a group
\begin{equation*}
D_{\Gamma} = \Lambda^{*}/2\Lambda^{*}\approx \mathbb{Z}_2^{n-k}.
\end{equation*}
Let us denote by $T_{\Gamma}$ an $(n-k)$-dimensional torus
\begin{equation}\label{torus}
T_{\Gamma} = (e^{i\pi(\gamma_1,\varphi)},...,e^{i\pi(\gamma_n,\varphi)}) \subset \mathbb{C}^n,
\end{equation}
where $\varphi=(\varphi_1,...,\varphi_{n-k})\in \mathbb{R}^{n-k}$ and $(\gamma_j,\varphi)=\gamma_{j,1}\varphi_1+...+\gamma_{j,(n-k)}\varphi_{n-k}$.

Consider a map
\begin{equation*}
\begin{gathered}
\widetilde{\psi}: \mathcal{R}\times T_{\Gamma} \rightarrow \mathbb{C}^n,\\
\widetilde{\psi}(u_1,...,u_n,\varphi) = (u_1e^{i\pi(\gamma_1,\varphi)},...,u_ne^{i\pi(\gamma_n,\varphi)}).
\end{gathered}
\end{equation*}

Let $\varepsilon \in D_{\Gamma}$ be a nontrivial element. We see that if $(u_1,...,u_n) \in \mathcal{R}$, then $(u_1\cos\pi(\varepsilon,\gamma_1),..,u_n\cos\pi(\varepsilon,\gamma_n))\in \mathcal{R}$ because $\cos\pi(\varepsilon,\gamma_i)=\pm 1$. We get that
\begin{equation*}
\widetilde{\psi}(u_1,...,u_n,\varphi) = \widetilde{\psi}(u_1\cos\pi(\varepsilon,\gamma_1),..,u_n\cos\pi(\varepsilon,\gamma_n), \varphi+\varepsilon).
\end{equation*}
Let us consider the quotient of $\mathcal{R}\times{T_{\Gamma}}$ by group $D_{\Gamma}$
\begin{equation}\label{mainmanifold}
\begin{gathered}
\mathcal{N} = \mathcal{R}\times_{D_{\Gamma}}T_{\Gamma},\\
(u_1,...,u_n,\varphi) \sim (u_1\cos\pi(\varepsilon,\gamma_1),..,u_n\cos\pi(\varepsilon,\gamma_n), \varphi + \varepsilon).
\end{gathered}
\end{equation}
The action of $D_{\Gamma}$ is free, since it is free on the second factor. Hence, $\mathcal{N}$ is a smooth $n$-manifold. So, we have a well-defined map
\begin{equation}\label{mainmap}
\begin{gathered}
\psi: \mathcal{N} \rightarrow \mathbb{C}^n,\\
\psi(u_1,...,u_n,\varphi) = (u_1e^{i\pi(\gamma_1,\varphi)},...,u_ne^{i\pi(\gamma_n,\varphi)}).
\end{gathered}
\end{equation}
Let us define an $(n-k)$-dimensional vector
\begin{equation}\label{maslovclass}
\gamma_1+...+\gamma_n = (t_1,...,t_{n-k})^T.
\end{equation}
\begin{theorem} (Mironov \cite{MironovCn}). The map $\psi$ is an immersion and the image is H-minimal Lagrangian. Moreover,
\begin{equation}\label{lagrmasl}
\frac{1}{\pi}\omega_{H} = t_1d\varphi_1+...+t_{n-k}d\varphi_{n-k},
\end{equation}
where $\varphi_i$ are coordinates on the torus as in $(\ref{torus})$ and $\omega_H$ is defined in $(\ref{meancurve})$.
\end{theorem}

\textbf{Remark.} In fact, Mironov found the Lagrangian angle and proved that the constructed submanifolds are H-minimal. But to simplify our paper we are not giving the definition of the Lagrangian angle.

This theorem was proved by Mironov in 2003. All details of the proof can be found in \cite{MironovCn}. Another point of view can be found in \cite{kotel}.
\\

It was discussed in section $\ref{definitions}$ that $\frac{1}{\pi}\omega_{H}$ is equal to the Maslov class.
\\
\\
\textbf{Example.} Let us consider the quadric
\begin{equation*}
\gamma_1u_1^2 + u_2^2 + ... + u_{2n-1}^2 + u_{2n}^2 = \delta_1, \quad \gamma_1 \in \mathbb{Z}_{>0},
\end{equation*}
which defines a manifold diffeomorphic to $S^{2n-1}$. Then, $\mathcal{N}=S^{2n-1} \times_{D_\Gamma} S^1$. The lattice $\Lambda \subset \mathbb{R}$ is generated by numbers $\gamma_1$ and $1$. We can choose $1$ as a basis for $\Lambda $ and $\Lambda^{*}$.  A nontrivial element $\varepsilon \in D_{\Gamma} \approx \mathbb{Z}_2$ acts on $S^{2n-1}$ by
\begin{equation*}
\varepsilon(u_1, u_2, ..., u_{2n}) = ((-1)^{\gamma_1}u_1, -u_2, ..., -u_{2n}).
\end{equation*}
We can cut $S^1$ into two halves and assume that each part is  segment $[0,1]$. So, $\mathcal{N}$ is obtained from the cylinder $S^{2n-1} \times [0,1]$ by identification of points on the boundary, i.e. $(y,0) \sim (\varepsilon(y),1)$. It is easy to see that if $\gamma_1$ is odd, then $\varepsilon$ preserves the orientation. If $\gamma_1$ is even, then $\varepsilon$ doesn't preserve the orientation. We obtain that
\begin{equation*}
\begin{gathered}
\mathcal{N} = S^{2n-1} \times S^1 \quad if \quad \gamma_1 \quad is \quad odd \\
\mathcal{N} = K^{2n} \quad  if \quad \gamma_1 \quad is \quad even
\end{gathered}
\end{equation*}
where $K^{2n}$ is the generalized  Klein bottle. We have that
\begin{equation*}
\frac{1}{\pi}\omega_{H} = (\gamma_1 + (2n-1))d\varphi
\end{equation*}
and $\psi(\mathcal{N})$ is an immersed submanifold of $\mathbb{C}^{2n}$.

\vspace{.28in}

Let us note that $D_{\Gamma}$ acts freely on $T_{\Gamma}$. Therefore, the projection
\begin{equation*}
\mathcal{N} = \mathcal{R}\times_{D_{\Gamma}} T_{\Gamma} \rightarrow T_{\Gamma}/D_{\Gamma}
\end{equation*}
onto the second factor is a fiber bundle with fibre $\mathcal{R}$ over $(n-k)$-dimensional torus $T_{\Gamma}/D_{\Gamma} = T^{n-k}$.

Let us denote by $\mathbb{Z}<\gamma_1,....,\gamma_n>$ the set of integer linear combinations of vectors $\gamma_1, ..., \gamma_n$. For any $u=(u_1,...,u_n) \in \mathcal{R}$ we have a sublattice
\begin{equation*}
\Lambda_u = \mathbb{Z}<\gamma_k : u_k \neq 0> \subset \Lambda = \mathbb{Z}<\gamma_1,...,\gamma_n>.
\end{equation*}

\begin{lemma}\label{l}(\cite{Mirpan} Theorem 4.1). The map $\psi$ defines an embedding if and only if $\Lambda_u = \Lambda$ for any $u \in \mathcal{R}$.
\end{lemma}

Lemma $\ref{l}$ says that in the example considered above manifold $\psi(\mathcal{N})$  is embedded if and only if $\gamma_1 = 1$. Indeed, $\Lambda_{(\sqrt{\frac{\delta_1}{\gamma_1}},0,...,0)} \neq \Lambda$ if $\gamma_1 \neq 1$.
\\

What can we say about the topology of $\mathcal{N}$? In fact, the topology of $\mathcal{N}$ can be highly complicated. In section $\ref{torictopology}$ we use methods of toric topology to study $\mathcal{N}$.

\section{Toric topology and intersection of quadrics}\label{torictopology}

In this section we discuss toric topology and its applications. For more details we refer our reader to paper of Panov \cite{Panov} (Sections 2,3,12). Much more details can be found in book \cite{torictop}.
\\

A convex \emph{polyhedron} $P$ is an intersection of finitely many halfspaces in $\mathbb{R}^k$. Bounded polyhedra are called \emph{polytopes}.

A \emph{supporting hyperplane} of $P$ is a hyperplane $H$ which has common points with $P$ and for which the polyhedron is contained in one of the two closed half-spaces determined by $H$. The intersection $P\cap H$ with a supporting hyperplane is called a face of the polyhedron. Zero-dimensional faces are called \emph{vertices}, one-dimensional faces are called \emph{edges}, and faces of codimension one are called \emph{facets}.

Consider a system of $n$ linear inequalities defining a convex polyhedron in $\mathbb{R}^k$
\begin{equation}\label{polytope}
P_{A,b}=\{x\in \mathbb{R}^k: <a_i,x>+b_i \geqslant 0 \quad for \quad i=1,...,n  \},
\end{equation}
where $<\cdot,\cdot>$ is the standard scalar product on $\mathbb{R}^k$, $a_i \in \mathbb{R}^k$, and $b_i \in \mathbb{R}$. By $b$ denote a vector $b=(b_1,...,b_n)^T$, $x=(x_1,...,x_k)^T$ and by $A$ the $k\times n$ matrix whose columns are the vectors $a_i$. Then, our polyhedron can be written in the following form:
\begin{equation*}
P_{A,b}=\{x\in \mathbb{R}^k: (A^Tx + b)_i \geqslant 0 \quad for \quad i=1,...,n\}.
\end{equation*}

\textbf{Definition.}  We say that (\ref{polytope}) is \emph{simple} if exactly $k$ facets meet at each vertex. We say that (\ref{polytope}) is \emph{generic} if for any vertex $x\in P$ the normal vectors $a_i$ of the hyperplanes containing $x$ are linearly independent.
\\

Assume that the vectors $a_1,...,a_n$ span $\mathbb{R}^k$.  By definition, we put
\begin{equation}\label{polytope2}
\begin{gathered}
i_{A,b} : \mathbb{R}^k \rightarrow \mathbb{R}^n, \\
i_{A,b}(x)=A^Tx + b = (<a_1, x> + b_1,...,<a_n, x> + b_n)^T.
\end{gathered}
\end{equation}
Then, the image $i_{A,b}(\mathbb{R}^k)$ is given by
\begin{equation}\label{polsys}
\begin{gathered}
i_{A,b}(\mathbb{R}^k)= \{u \in \mathbb{R}^n : \Gamma u = \Gamma b \},\\
\Gamma A^T = 0, \quad u=(u_1,...,u_n)^T,
\end{gathered}
\end{equation}
where $\Gamma $ is $(n-k)\times n$-matrix whose rows form a basis of linear relations between the vectors $a_i$. The set of columns $\gamma_1,...,\gamma_n$ of $\Gamma$ is called a \emph{Gale dual} configuration of $a_1,...,a_n$. Each of the matrices $A$ and $\Gamma$ determines the other uniquely up to multiplication by an invertible matrix from the left. We have
\begin{equation}\label{polytopeemb}
i_{A,b}(P) = i_{A,b}(\mathbb{R}^k)\cap \mathbb{R}^n_{+}.
\end{equation}

Let us describe the correspondence between the intersection of quadrics and polyhedra. Replacing $u_i$ by $u_i^2$ in (\ref{polsys}) we get $(n-k)$ quadrics which define a subset in $\mathbb{R}^n$.

Now assume that we have
\begin{equation}\label{quadrics}
\mathcal{R}_{\Gamma, \delta} = \{u\in \mathbb{R}^n : \gamma_{1,i}u_1^2 + ... + \gamma_{n,i}u_n^2 = \delta_i, \quad i=1,...,n-k, \}.
\end{equation}
The coefficients of the quadrics define $(n-k)\times n$ matrix $\Gamma=(\gamma_{j,k})$. The group $\mathbb{Z}_{2}^n$ acts on $\mathcal{R}_{\Gamma, \delta}$ by
\begin{equation*}
\varepsilon\cdot (u_1,...,u_n) = (\pm u_1,...,\pm u_n),
\end{equation*}
The quotient $\mathcal{R}_{\Gamma, \delta}/\mathbb{Z}_{2}^n$ can be identified with the set of nonnegative solutions of the system
\begin{equation*}
\left\{
 \begin{array}{l}
\gamma_{1,1}u_1 + ... + \gamma_{n,1}u_n = \delta_1\\
\ldots \quad \quad \\
\gamma_{1,(n-k)}u_1 + ... + \gamma_{n,(n-k)}u_k = \delta_{n-k}
 \end{array}
\right.
\end{equation*}
And we get the same system as in (\ref{polsys}) and (\ref{polytopeemb}). Solving the homogeneous version of the system above we get the matrix $A$. So, rows of matrix $\Gamma$ form a basis of linear relations between the vectors $a_i$. Then, we can construct a polytope (\ref{polytope}), where $(b_1,...,b_n)$ is an arbitrary solution of the linear system above.

We obtain that a polyhedron defines a system of quadrics and a system of quadrics defines a polyhedron.
\\

\textbf{Definition.} It may happen that some of the inequalities can be removed from the presentation without changing $P_{A,b}$. Such inequalities are called \emph{redundant}. A presentation without redundant inequalities is called \emph{irredundant}.

\begin{theorem}(\cite{Panov} Theorem 3.5 and Chapter 12, \cite{Mirpan}).
\\
\textbf{1)} Assume that we have a polyhedron defined by
\begin{equation*}
P_{A,b}=\{x\in \mathbb{R}^k: \quad <a_i,x>+b_i \geqslant 0 \quad for \quad i=1,...,n  \},
\end{equation*}
where $a_1,...,a_n$ span $\mathbb{R}^k$. And
\begin{equation*}
\mathcal{R}_{\Gamma, \delta} = \{u\in \mathbb{R}^n : \gamma_{1,i}u_1^2 + ... + \gamma_{n,i}u_n^2 = \delta_i, \quad i=1,...,n-k, \}
\end{equation*}
is the corresponding intersection of quadrics. Then columns of the system $\gamma_1,...,\gamma_n$ span $\mathbb{R}^{n-k}$. The intersection of quadrics is defined uniquely up to a linear isomorphism of $\mathbb{R}^{n-k}$, and $\mathcal{R}_{\Gamma, \delta}$ defines $P_{A,b}$ uniquely up to an isomorphism of $\mathbb{R}^k$. Also, $\Gamma b = \delta$.

\textbf{2)} The intersection of quadrics $\mathcal{R}_{\Gamma, \delta}$ is nonempty and smooth if and only if the presentation $P_{A,b}$ is generic.

\textbf{3)} The intersection of quadrics $\mathcal{R}_{\Gamma, \delta}$ is connected if and only if the presentation $P_{A,b}$ is irredundant.
\end{theorem}

\textbf{Definition}. Let us assume that $\mathbb{Z}<a_1,....,a_n>$ defines a lattice, where $\mathbb{Z}<a_1,....,a_n>$ is the set of integer linear combinations of vectors $a_1, ..., a_n$. Polyhedron $P$ is called \emph{Delzant} if it is simple and for any vertex $x \in P$ the vectors $a_i$ normal to the facets meeting at $x$ form a basis for the lattice $\mathbb{Z}<a_1,....,a_n>$.
\\

\textbf{Definition}. A Delzant polytope $P$ is called \emph{Fano} if it can be defined by
\begin{equation*}
P_{A,b}=\{x\in \mathbb{R}^k: \langle a_i,x \rangle + c \geq  0 \quad for \quad i=1,...,n  \},
\end{equation*}
where  each vector $a_i \in \mathbb{Z}^k$ is the primitive integral interior normal to the corresponding facet. In other words, $c=b_1=...=b_n$.
\\

In the previous section we constructed map $\psi : \mathcal{N}=\mathcal{R}\times_{D_{\Gamma}} T_{\Gamma} \rightarrow \mathbb{C}^n$ using intersection of quadrics. Lemma $\ref{l}$ says that $\psi$ is  an embedding if and only if $\Lambda_u = \Lambda$. Lemma $\ref{l}$ is equivalent to the following theorem:

\begin{theorem}\label{embdelzant}(\cite{Mirpan})
The map $\psi$ defines an embedding if and only if the polyhedron $P$ corresponding to system (\ref{equation}) is Delzant.
\end{theorem}

What can we say about the topology of $\mathcal{R}$? First, let us mention the following lemma:

\begin{lemma}\label{lemsurface}(see \cite{torictop} Proposition 4.1.8) System of quadrics associated to $m-$gon defines an oriented surface of genus $g = 1 + (m-4)2^{m-3}$, where $m \geqslant 5$.
\end{lemma}

In the case of three quadrics, the topology of $\mathcal{R}$ was fully described in \cite{Lopez2}. Assume that we have
\begin{equation*}
\left\{
 \begin{array}{l}
 u_1^2 + ... + u_n^2 = 1 \\
\sum\limits_{i=1}^{n}a_i u_i^2 = 0 \\
\sum\limits_{i=1}^{n}b_i u_i^2 = 0\\
 \end{array}
\right.
\end{equation*}
where $a_{i}, b_{i} \in \mathbb{R}$. Suppose that the system above is regular, i.e. defines a smooth manifold $\mathcal{R}$. It turns out that the system is regular if $0 \in \mathbb{R}^2$ doesn't belong to the line interval connecting any two of the $\lambda_i = (a_i, b_i)$. When we move the points $\lambda_i$ around $\mathbb{R}^2$ without breaking the regularity condition ($0 \in \mathbb{R}^2$ doesn't belong to the line interval connecting any two of the $\lambda_i$), then we don't change the diffeomorphism type of $\mathcal{R}$. We can join together as many $\lambda_i$ as possible (without breaking the regularity condition) in single points.

\begin{figure}[h]
\includegraphics[width=1\linewidth]{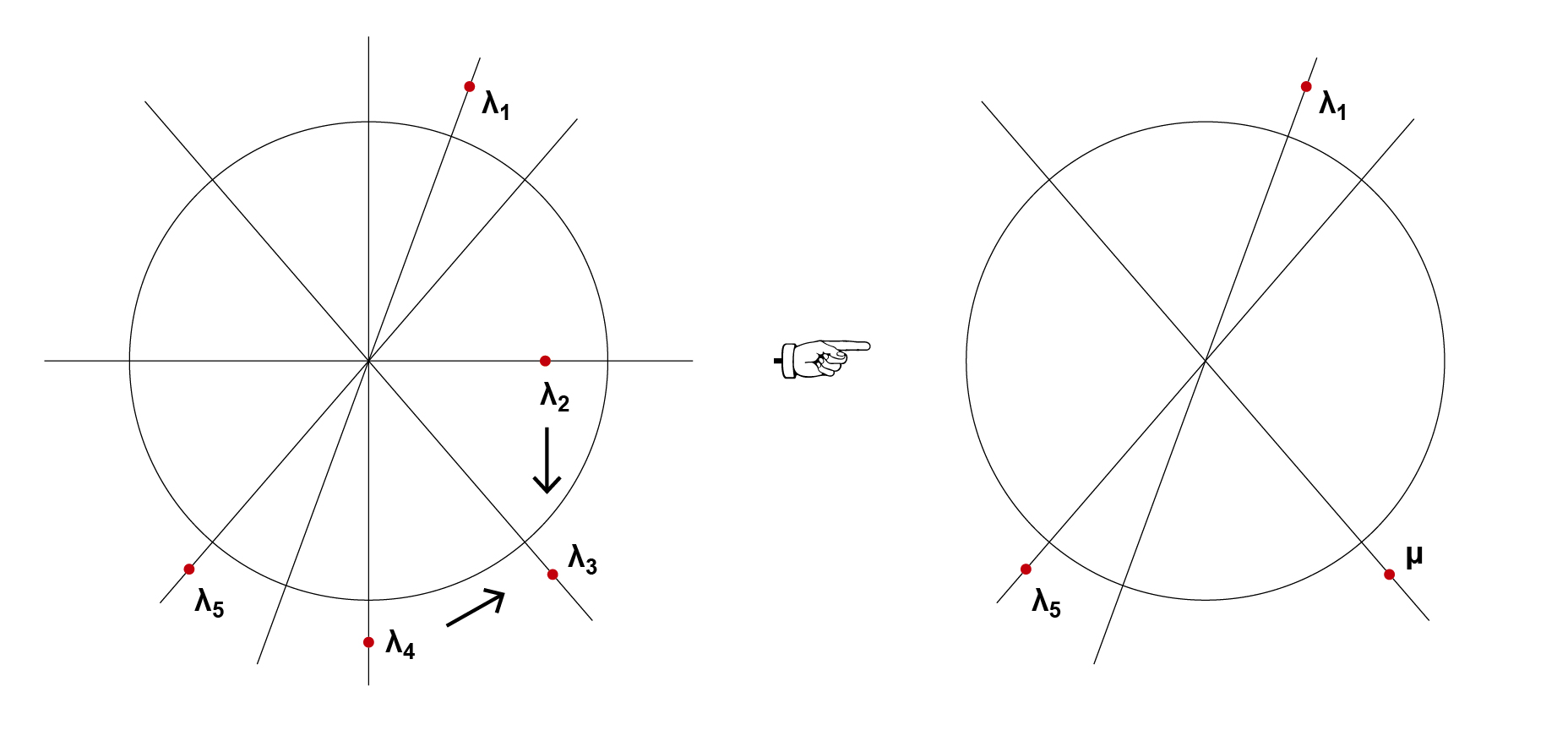}
\end{figure}

For example, in the picture  we can join points $\lambda_2, \lambda_3, \lambda_3$ together in a single point. We can't joint points $\lambda_4, \lambda_5$ without breaking the regularity condition because they are separated by the line connecting $\lambda_1$ and $0$. If $\lambda_4$ crosses the line connecting $\lambda_1$ and $0$, then $0$ belongs to the line connecting $\lambda_4, \lambda_1$ and this breaks the regularity condition. As a result we get three points.

Then let us push  the points radially until they are in the unit circle and distribute them evenly along the circle. It appears that  after this process we always get odd number of points.  So, the $\lambda_i$ can be assumed to be $(2l+1)-$th roots of unity and let us denote roots by $\rho^1,...,\rho^{2l+1}$, where $\rho^j = e^{\frac{2j\pi}{2l+1}}$. Note that vectors $\rho^j$ comes with multiplicity $n_j$ (number of joined vectors $\lambda_i$ during the deformation), where $n_j > 0$. For example, in the right picture points $\lambda_1, \lambda_5$ come with multiplicity 1 and point $\mu$ comes with multiplicity $3$ (we joined three points $\lambda_2, \lambda_3, \lambda_4$ to get $\mu$).

We get that the topology of $\mathcal{R}$ is described by numbers $n_1,...,n_{2l+1}$. Denote
\begin{equation*}
d_i = n_i + ... + n_{i+l-1}, \quad i=1,...,2l+1,
\end{equation*}
where $j$ in $n_j$ is reduced modulo $2l+1$ if $j > 2l+1$.

\begin{theorem}\label{threequad}(see \cite{Lopez2} for more general result) Let $\mathcal{R}$ be the variety corresponding to $n = n_1 + n_2 + ... +n_{2l+1}$. \\
1) if $l=1$, then $\mathcal{R}$ is diffeomorphic to the product $S^{n_1-1} \times S^{n_2 - 1} \times S^{n_3-1}$,\\
2) If $l > 1$, then $\mathcal{R}$ is diffeomorphic to the connected sum
\begin{equation*}
\#_{i=1}^{2l+1} (S^{d_i -1} \times S^{n-d_i -2}).
\end{equation*}
\end{theorem}

Let $P$ be an $n$-polytope and $H$ be a hyperplane  that does not contain any vertex of $P$. Then the intersections $P \cap \{H \leqslant 0 \}$, $P \cap \{H \geqslant 0 \}$ are simple polytopes.  If $H$ separates a vertex $v$ from the other vertices of $P$ and $v \subset \{ H \leqslant 0 \}$, then we say that the polytope $P \cap \{H \geqslant 0 \}$ is obtained from $P$ by a vertex truncation. Let us denote by $P_v$ and $\mathcal{R}_v$ the polytope obtained by vertex truncation and the corresponding system of quadrics, respectively.

\begin{theorem}\label{Lop}(\cite{Lopez3}) Let $P$ be a simple polytope of dimension $n$ with $k$ facets. Then $\mathcal{R}_v$ is diffeomorphic to $\mathcal{R} \# \mathcal{R} \#_{(2^{k-n} - 1)}(S^1 \times S^{n-1})$.
\end{theorem}

Also, let us mention the following theorem:

\begin{theorem}\label{simp}(\cite{Bahri}, \cite{Lopez}) The intersection of quadrics $\mathcal{R}$ is $(k-1)$-connected if, and only if, the intersection of any $k$ facets of corresponding polytope $P$ is non-empty.
\end{theorem}

From Theorem $\ref{Lop}$ we see that $\mathcal{R}_v$ is not simply-connected.

\section{Proof of Theorem $\ref{Fanos}$}\label{pfan}

It is shown in the previous section that any generic polytope $P$ corresponds to the smooth manifold $\mathcal{R}$, where $\mathcal{R}$ is defined by the system of quadrics. In section $\ref{Ham}$ we showed that the system of quadrics corresponds to the immersed Lagrangian submanifold $L \subset \mathbb{C}^n$. From Theorem $\ref{embdelzant}$ we know that if $P$ is Delzant, then $L$ is embedded. Moreover, we have the following theorem:
\\
\\
\textbf{Theorem $\ref{Fanos}$}. Let $P$ be a Delzant polytope and $L \subset \mathbb{C}^n$ be the corresponding embedded Lagrangian. Assume that $P$ is irredundant (or equivalently $\mathcal{R}$ is connected). Then $L$ is monotone if and only if the polytope $P$ is Fano.
\\
\begin{proof}

Let us recall some notations. We consider $\mathbb{C}^n$ with the standard symplectic form
\begin{equation*}
\omega = dx_1\wedge dy_1 + ... + dx_n\wedge dy_n.
\end{equation*}
The Liouville form is given by
\begin{equation*}
\lambda = x_1dy_1 + ... + x_ndy_n, \quad d\lambda = \omega.
\end{equation*}
The embedding of $L$ is given by formula $(\ref{mainmap})$
\begin{equation*}
\begin{gathered}
L = \psi(\mathcal{N}) \subset \mathbb{C}^n, \quad  \mathcal{N} = \mathcal{R} \times_{D_{\Gamma}} T_{\Gamma},\\
\psi(u_1,...,u_n,\varphi) = (u_1e^{i\pi(\gamma_1,\varphi)},...,u_ne^{i\pi(\gamma_n,\varphi)}).
\end{gathered}
\end{equation*}
As in section $\ref{Ham}$, let $\Lambda \subset \mathbb{R}^{n-k}$ be the lattice generated by columns $\gamma_1,...,\gamma_n \in  \mathbb{Z}^{n-k}$,  $\Lambda^{*}$ be the dual lattice and $D_{\Gamma}$ be the group $\Lambda^{*}/2\Lambda^{*} = \mathbb{Z}_2^{n-k}$.

Let $b$ be an arbitrary element of $H_1(\mathcal{N}, \mathbb{Z})$ and $pr$ be the projection
\begin{equation*}
pr: \mathcal{R} \times T_{\Gamma} \rightarrow \mathcal{N} = \mathcal{R} \times_{D_{\Gamma}} T_{\Gamma}.
\end{equation*}
We can think about $\mathcal{R} \times T_{\Gamma}$ as $2^{n-k}$ sheeted covering space of $\mathcal{N}$. We have the induced map on singular chains $pr: C_{1}(\mathcal{R} \times T_{\Gamma}) \rightarrow C_1(\mathcal{N})$ and there is also homomorphism in the opposite direction $\tau: C_1(\mathcal{N}) \rightarrow C_1(\mathcal{R} \times T_{\Gamma})$ which assigns to a singular simplex the sum of the $2^{n-k}$ distinct lifts. This is a chain map and we have $\tau_{*}:H_1(\mathcal{N}, \mathbb{Z}) \rightarrow H_1(\mathcal{R} \times T_{\Gamma}, \mathbb{Z})$, $pr \circ \tau(b) = 2^{n-k}b$. If we consider homologies with real coefficients, then we get $pr \circ \tau(\frac{b}{2^{n-k}}) = b$. In other words, $pr_{*}: H_{1}(\mathcal{R} \times T_{\Gamma}, \mathbb{R}) \rightarrow H_1(\mathcal{N}, \mathbb{R})$ is surjective.

Assume $H_1(\mathcal{R}, \mathbb{R})$ is generated by elements $\tilde{d}_1,...,\tilde{d}_m$ and $H_1(\mathbb{T}_{\Gamma}, \mathbb{R}) = \mathbb{R}^{n-k}$ is generated by elements $\tilde{e}_1,...,\tilde{e}_{n-k}$. Then we have that $H_1(\mathcal{R}\times T_{\Gamma} , \mathbb{R})= H_1(\mathcal{R}, \mathbb{R}) \oplus H_1(T_{\Gamma}, \mathbb{R})$ is generated by $\tilde{d}_1,...,\tilde{d}_m,\tilde{e}_1,...,\tilde{e}_{n-k}$ and $H_1(\mathcal{N}, \mathbb{R})$ is generated by $d_1 = pr_{*}(\tilde{d}_1),..., d_m = pr_{*}(\tilde{d}_m), e_1 = pr_{*}(\tilde{e}_1),..., e_{n-k} = pr_{*}(\tilde{e}_{n-k})$ because $pr_{*}: H_{1}(\mathcal{R} \times T_{\Gamma}, \mathbb{R}) \rightarrow H_1(\mathcal{N}, \mathbb{R})$ is surjective.

For simplicity let us denote the elements $\psi_{*}(e_i)$, $\psi_{*}(d_j) \in H_1(L, \mathbb{R})$ by $e_i, d_j \in H_1(\mathcal{N}, \mathbb{R})$ respectively.

Without loss of generality, assume that the vectors $\gamma_1$,...,$\gamma_{n-k}$ form a basis for $\Lambda$. Let $\varepsilon_1,...,\varepsilon_{n-k}$ be the vectors dual to $\gamma_1$,...,$\gamma_{n-k}$. Let us find $\tilde{e}_i$ and $e_i$ explicitly. Easy to see that
\begin{equation}\label{cycles}
\begin{gathered}
\tilde{e}_i: I \rightarrow \mathcal{R} \times T_{\Gamma} \\
\tilde{e}_i(s) = (u_1,...,u_n, 2s\varepsilon_i),\\
e_i: I \rightarrow L = \psi(\mathcal{N}) \\
e_i(s) = \psi(u_1,...,u_n, 2s\varepsilon_i),\\
I = [0,1], \;\; s\in [0,1], \;\; u_1,...,u_n=const, \;\; i=1,...,n-k
\end{gathered}
\end{equation}
represent $1-$cycles, $\tilde{e}_1,...,\tilde{e}_{n-k}$ generate $H_1(T_{\Gamma}, \mathbb{R})$ and $e_i = pr_{*}(\tilde{e}_i)$.

By $\varepsilon_{i,p}$ denote the $p$th coordinate of $\varepsilon_i$.

\begin{lemma}\label{areaform}
We have
\begin{equation}\label{area}
\begin{gathered}
\psi^{*}(\lambda)(e_i) = \pi\sum\limits_{p=1}^{n-k}\varepsilon_{i,p}\delta_p,\\
\psi^{*}(\lambda)(d_j) = 0,
\end{gathered}
\end{equation}
where $\delta_p$ defined in $(\ref{equation})$, $i=1,...,n-k$, and $j=1,...,m$.
\end{lemma}
\begin{proof}
From (\ref{mainmap}) we have
\begin{equation*}
x_j = u_j\cos(\pi<\gamma_j,\varphi>), \quad y_j = u_j\sin(\pi<\gamma_j,\varphi>),
\end{equation*}
where $\varphi = (\varphi_1,...,\varphi_{n-k})$. Direct calculations show
\begin{equation*}
\begin{gathered}
\psi^{*}dy_j = \sin(\pi<\gamma_j,\varphi>)du_j + \pi\sum\limits_{p=1}^{n-k} u_j\gamma_{j,p}\cos(\pi<\gamma_j,\varphi>)d\varphi_p,
\\
\psi^{*}(x_jdy_j) = \frac{1}{2}\sin(2\pi<\gamma_j,\varphi>)u_jdu_j + \pi\sum\limits_{p=1}^{n-k} u_j^2\gamma_{j,p}\cos^2(\pi<\gamma_j,\varphi>)d\varphi_p = \\
\\
\frac{1}{4}\sin(2\pi<\gamma_j,\varphi>)d(u_j^2) + \pi\sum\limits_{p=1}^{n-k} u_j^2\gamma_{j,p}\cos^2(\pi<\gamma_j,\varphi>)d\varphi_p.
\end{gathered}
\end{equation*}
As a result we obtain
\begin{equation}\label{areaprep}
\begin{gathered}
\psi^{*}(\lambda) = \frac{1}{4}\sum\limits_{j=1}^{n}\sin(2\pi<\gamma_j,\varphi>)d(u_j^2) +
\\
 \pi\sum\limits_{j=1}^{n}\sum\limits_{p=1}^{n-k} u_j^2\gamma_{j,p}\cos^2(\pi<\gamma_j,\varphi>)d\varphi_p =
\\
\frac{1}{4}\sum\limits_{j=1}^{n}\sin(2\pi<\gamma_j,\varphi>)d(u_j^2) + \pi\sum\limits_{p=1}^{n-k} \delta_p\cos^2(\pi<\gamma_j,\varphi>)d\varphi_p.
\end{gathered}
\end{equation}
We see from (\ref{mainmanifold}) that the forms $d\varphi_1,...,d\varphi_{n-k}$ are closed forms and invariant under the action of $D_{\Gamma}$ . Therefore, they are elements of $H^1(\mathcal{N}, \mathbb{R})$. We get
\begin{equation*}
\begin{gathered}
\psi^{*}(\lambda)(e_i) = \int\limits_{e_i}\psi^{*}(\lambda) =  \pi \sum\limits_{p=1}^{n-k}\varepsilon_{i,p}\delta_p.
\end{gathered}
\end{equation*}

Let us note that cycles $\tilde{d}_j$ and $d_j = pr(\tilde{d}_j)$ can be represented by
\begin{equation*}
\begin{gathered}
\tilde{d}_j: I \rightarrow \mathcal{R} \times T_{\Gamma} \\
\tilde{d}_j(s) = (u_1(s),...,u_n(s), 0),\\
d_j: I \rightarrow L = \psi(\mathcal{N}) \\
d_j(s) = \psi(u_1(s),...,u_n(s), 0),\\
I = [0,1], \;\; s\in [0,1].
\end{gathered}
\end{equation*}
In other words, we can assume that $\varphi_1 = ... =\varphi_{n-k} = 0$. From $(\ref{areaprep})$ we see that if $\varphi_1 = ... =\varphi_{n-k} = 0$, then $\psi^{*}(\lambda) = 0$  ( and is exact for other values of $\varphi_i$ because $u_i^2$ is well defined function on $\mathcal{N}$). This implies $\psi^{*}(\lambda)(d_j)=0$.

\end{proof}

\vspace{.05in}

Note that if $b$ belongs to the torsion of $H_1(\mathcal{N}, \mathbb{Z})$, then $\psi^{*}(\lambda)(b) = 0$.

Let us consider a path $r_i(s) \subset \mathcal{N}$ such that
\begin{equation*}
r_i(0) = (u_1,...,u_n,0,...,0), \; \; r_i(1) = (u_1\cos\pi \langle \varepsilon_i, \gamma_1 \rangle,....,u_n\cos\pi \langle \varepsilon_i, \gamma_n \rangle,\varepsilon_i),\\
\end{equation*}
where $i=1,...,n-k$. The path above exists because $\mathcal{R}$ is connected. Note that $r_i(1) = \varepsilon_i \cdot r_i(0)$. Hence, $r_i(s)$ is a loop.

\begin{lemma}\label{class}
We have $\psi^{*}(\lambda)(e_i) = 2\psi^{*}(\lambda)(r_i)$ and $r_i$ is primitive element of $H_1(\mathcal{N}, \mathbb{Z})$.
\end{lemma}
\begin{proof}
 As we noticed in the end of Section $\ref{Ham}$, $\mathcal{N}$ fibers over $T_{\Gamma}/D_{\Gamma} = T^{n-k}$ with fiber $\mathcal{R}$. Let $\sigma$ be the projection $\mathcal{N} = \mathcal{R} \times_{D_\Gamma} T_{\Gamma} \rightarrow T_{\Gamma}/D_\Gamma = T^{n-k}$. From the long exact sequence
\begin{equation*}
0 \rightarrow \pi_1(\mathcal{R}) \rightarrow \pi_1(\mathcal{N}) \rightarrow \pi_1(T^{n-k})
\end{equation*}
we see that cycles $\sigma_{*}(r_1), ..., \sigma_{*}(r_{n-k})$ generate $\pi_1(T^{n-k})$ and this means that they are primitive elements of $H_1(\mathcal{N}, \mathbb{Z})$. Also, $\sigma_{*}(e_i) = 2\sigma_{*}(r_i)$, $\sigma_{*}(d_j) = 0$ for any $i=1,...,n-k$ and $j=1,...,m$. Let us recall that $H_1(\mathcal{N}, \mathbb{R})$ is generated by $d_1,...,d_m,e_1,...,e_{n-k}$. Therefore,
\begin{equation*}
r_i = \sum\limits_{j=1}^m k_j d_j  + \frac{e_i}{2}
\end{equation*}
where $k_j$ are some real numbers. Then from Lemma $\ref{areaform}$ we have $\psi^{*}(\lambda)(r_i) = \frac{1}{2}\psi^{*}(\lambda)(e_i)$.
\end{proof}

Let $\mu$ be the Maslov class of $L = \psi(\mathcal{N})$. Arguing as before and using $(\ref{lagrmasl})$ we can prove that
\begin{equation}\label{lagrmasl2}
\begin{gathered}
\mu(r_i) = \frac{1}{2}\mu(e_i) = \int\limits_0^1 \sum\limits_{p=1}^{n-k}\varepsilon_{i,p}t_p ds =  \sum\limits_{p=1}^{n-k}\varepsilon_{i,p}t_p, \\
\mu(d_j) = 0
\end{gathered}
\end{equation}
where numbers $t_p$ are defined in $(\ref{maslovclass})$ and $\varepsilon_{i,p}$ is the $p$th coordinate of $\varepsilon_i$.
\\

Finally, let us prove our theorem. If $P$ is Fano, then $P$ is given by
\begin{equation*}
P_{A,b}=\{x\in \mathbb{R}^k: <a_i,x> + c \geq  0 \quad for \quad i=1,...,n  \}.
\end{equation*}
In other words (see Section $\ref{torictopology}$  for definitions)
\begin{equation*}
b = (\underbrace{c,...,c}_n)^T.
\end{equation*}
From $(\ref{polsys})$ we see that the corresponding system of quadrics has the form
\begin{equation*}
\begin{gathered}
\Gamma u = \Gamma b = \delta\\
\delta_p = c(\gamma_{1,p}+...+\gamma_{n,p}).
\end{gathered}
\end{equation*}
From $(\ref{area})$, $(\ref{lagrmasl2})$ we have
\begin{equation*}
\begin{gathered}
\psi^{*}(\lambda)(e_i) = c\pi \sum\limits_{p=1}^{n-k}\varepsilon_{i,p}(\gamma_{1,p} + ... + \gamma_{n,p}) = c\pi \sum\limits_{p=1}^{n-k}\varepsilon_{i,p}t_p, \\
\mu(e_i) = 2\int\limits_0^1(t_1\varepsilon_{i,1} + ... + t_{n-k}\varepsilon_{i, n-k})ds = 2 \sum\limits_{p=1}^{n-k}\varepsilon_{i,p}t_p.
\end{gathered}
\end{equation*}

We see $\mu(e_i) = \frac{2}{\pi c}\psi^{*}(\lambda)$ for all $i$.  From Lemma $\ref{class}$ and formula $(\ref{lagrmasl2})$ we obtain
\begin{equation*}
\mu(r_i) = \frac{1}{2}\mu(e_i), \quad \psi^{*}(\lambda)(r_i) = \frac{1}{2}\psi^{*}(\lambda)(e_i).
\end{equation*}
Finally, we obtain
\begin{equation*}
\mu = \frac{2}{c\pi}\psi^{*}(\lambda).
\end{equation*}
This means that $L$ is monotone and the first part of the theorem is proved
\\

Now let us assume that $L$ is monotone with monotonicity constant $\rho$. Hence,
\begin{equation*}
\mu(r_i) = \frac{1}{2}\mu(e_i) = \sum\limits_{p=1}^{n-k}\varepsilon_{i,p}t_p =  \frac{\rho}{2}\psi^{*}(\lambda)(e_i) = \frac{\rho}{2}\sum\limits_{p=1}^{n-k}\varepsilon_{i,p}\delta_p  = \rho \psi^{*}(\lambda)(r_i).
\end{equation*}
Then
\begin{equation*}
\sum\limits_{p=1}^{n-k}\varepsilon_{i,p}(\frac{\rho}{2}\delta_p - t_p) = 0   \Leftrightarrow   \frac{\rho}{2}\delta_p = t_p = \gamma_{1,p} + ... + \gamma_{n,p}
\end{equation*}
for any $i=1,...,n-k$. Let us recall that $\Gamma$ is the matrix with columns $\gamma_j$. We see that $b= (\frac{2}{\rho},...,\frac{2}{\rho})^T$ solves the equation $\Gamma b = \delta= (\delta_1,...,\delta_{n-k})^T$. Therefore, our polytope is Fano.

\end{proof}

\section{Proof of Theorem $\ref{ex1}$}

Let us consider a polytope $P_k \subset \mathbb{R}^{n-2}$ defined by inequalities

\begin{equation*}
\begin{gathered}
\left\{
 \begin{array}{l}
x_i + 1 \geqslant 0 \quad i=1,...,p-1 \\
-x_1 - ... - x_{p-1} +1 \geqslant 0 \\
x_i + 1 \geq 0 \quad i=p,...,n-2\\
-x_1 - ... - x_{k} - x_{p} - ... - x_{n-2} + 1 \geqslant 0  \\
 \end{array}
\right.
\\n-p+k > p, \; \; k < p-1, \;\;  p  \geqslant 1, \;\; n-p \geqslant 1
\end{gathered}
\end{equation*}
When $k=0$ we see that $P_k$ is product of $(p-1)-$simplex $\Delta^{p-1}$ and $(n-p-1)-$simplex $\Delta^{n-p-1}$. Varying $k$ we change the angle between $\Delta^{p-1}$ and $\Delta^{n-p-1}$.

From section $\ref{Ham}$ we have that the corresponding matrix $\Gamma$ and the system of quadrics have the following forms:

$ \quad \quad \quad \quad \quad  \Gamma = $\(
\begin{pmatrix}
    1 & ... & 1 & 1 & ...& 1 & 0 & ... & 0 \\
    1 & ... & 1 & 0 & ...& 0 & 1 & ... & 1 \\
\end{pmatrix}
\)
\begin{equation}\label{eq}
\begin{gathered}
\left\{
 \begin{array}{l}
u_1^2 +...+u_p^2 = p\\
u_1^2 + ... + u_{k}^2 + u_{p+1}^2 + ... + u_n^2 = n-p + k \\
 \end{array}
\right.
\\ n - p +k > p, \;\; k < p-1
\end{gathered}
\end{equation}
The system above is equivalent to
\begin{equation}\label{sys2}
\begin{gathered}
\left\{
 \begin{array}{l}
2u_1^2 + ... + 2u_k^2 + u_{k+1}^2 + ... + u_n^2 = n+k\\
(n-2p+k)u_1^2 + ... + (n-2p+k)u_k^2 + (n-p+k)u_{k+1}^2 + ... + (n-p+k)u_p^2 -\\
\qquad \qquad \qquad - pu_{p+1}^2 - ... - pu_n^2 = 0
 \end{array}
\right.  \\
\end{gathered}
\end{equation}
The second equation of system $(\ref{sys2})$ defines a cone over the product of two ellipsoids of dimensions $p - 1$ and $n-p-1$. By intersecting it with ellipsoid of dimension $n-1$, defined by the first equation, we obtain that the system defines
\begin{equation*}
\mathcal{R} = S^{p-1} \times S^{n-p-1} \subset \mathbb{R}^n.
\end{equation*}

From $(\ref{eq})$ we have that
\begin{equation*}
\gamma_1=...=\gamma_k = (1,1)^T, \quad \gamma_{k+1}=...=\gamma_p = (1,0)^T, \quad \gamma_{p+1} =...=\gamma_n = (0,1)^T.
\end{equation*}
Then the corresponding torus $T_{\Gamma}$ and the embedding of $\mathcal{N}(k,p,n) = \mathcal{R} \times_{D_{\Gamma}} T_{\Gamma}$ into $\mathbb{C}^n$ are given by
\begin{equation*}
T_{\Gamma} = (\underbrace{e^{i\pi(\varphi_1 + \varphi_2)},...,e^{i\pi(\varphi_1 + \varphi_2)}}_k, \underbrace{e^{i\pi\varphi_1},...,e^{i\pi\varphi_1}}_{p-k},\underbrace{e^{i\pi\varphi_2},...,e^{i\pi\varphi_2}}_{n-p}) \subset \mathbb{C}^n,
\end{equation*}
\begin{equation}\label{emb1}
\begin{gathered}
\psi: \mathcal{N}(k,p,n) \rightarrow \mathbb{C}^n \\
\psi(u_1,...,u_n, \varphi_1, \varphi_2) = \\
 (u_1e^{i\pi(\varphi_1 + \varphi_2)}, ...,u_ke^{i\pi(\varphi_1 + \varphi_2)}, u_{k+1}e^{i\pi\varphi_1},...,u_pe^{i\pi\varphi_1},u_{p+1}e^{i\pi\varphi_2},...,u_ne^{i\pi\varphi_2})\\
\varphi_1, \varphi_2 \in \mathbb{R}.
\end{gathered}
\end{equation}

The lattice $\Lambda$ is generated by vectors $\gamma_p$ and $\gamma_n$, the lattice $\Lambda^{*}$ is generated by $(1,0)$, $(0,1)$, and $D_{\Gamma} \approx \mathbb{Z}_2^2$.  We know that $\mathcal{R} \times_{D_{\Gamma}} T_{\Gamma} \rightarrow T^2 = T_{\Gamma}/D_{\Gamma}$ is a fibration, where the fiber is $\mathcal{R}$. We have the fibration over $T^2$ with fiber $S^{p-1} \times S^{n-p-1}$ and with transition maps $\varepsilon_p, \varepsilon_n \in D_{\Gamma}$
\begin{equation}\label{trans}
\begin{gathered}
\varepsilon_p(u_1,...,u_n ) \rightarrow  (-u_1,...,-u_{p}, u_{p+1},...,u_n ),\\
\varepsilon_n(u_1,...,u_n) \rightarrow  (-u_1,...,-u_{k}, u_{k+1},...,u_{p}, -u_{p+1},...,-u_n).
\end{gathered}
\end{equation}

\vspace{.1in}

\begin{lemma}\label{orientation} If $k,p,n$ are even numbers, then $L_k$ is diffeomorphic to $S^{p-1} \times S^{n-p-1} \times T^2$. The fibration is orientable if and only if numbers  $p$ and $n-p+k$ are even.
\end{lemma}

\begin{proof}

\textbf{1)}If $k,p,n$ are even numbers, then
\begin{equation*}
\begin{gathered}
(-u_1\cos\phi - u_2\sin\phi, u_1\sin\phi - u_2\cos\phi,...,\\
-u_{p-1}\cos\phi - u_{p}\sin\phi, u_{p-1}\sin\phi - u_{p}\cos\phi, u_{p+1},..., u_n),\\
 \phi \in [0,\pi]
\end{gathered}
\end{equation*}
belongs to $\mathcal{R}$ for all $\phi$. The expression above defines the isotopy between $\varepsilon_p|_{fiber}$ and identity map. In the same way the following expression

\begin{equation*}
\begin{gathered}
(-u_1\cos\phi - u_2\sin\phi, u_1\sin\phi - u_2\cos\phi,...,-u_{k-1}\cos\phi - u_k\sin\phi, \\
 u_{k-1}\sin\phi - u_k\cos\phi, u_{k+1}, ..., u_p,  \\
-u_{p+1}\cos\phi - u_{p+2}\sin\phi, u_{p+1}\sin\phi - u_{p+2}\cos\phi,..., -u_{n-1}\cos\phi - u_{n}\sin\phi, \\
 u_{n-1}\sin\phi - u_{n}\cos\phi)\\
 \phi \in [0,\pi]
\end{gathered}
\end{equation*}

defines isotopy between $\varepsilon_n|_{fiber}$ and identity map. We obtain that our fiber bundle is trivial and
\begin{equation*}
\mathcal{N}(k,p,n) = S^{p-1}\times S^{n-p-1}\times T^2.
\end{equation*}

\textbf{2)} If $k, p, n$ are odd numbers, then
\begin{equation*}
\begin{gathered}
(-u_1, -u_2\cos\phi - u_3\sin\phi, u_2\sin\phi - u_3\cos\phi,...,-u_{k-1}\cos\phi - u_{k}\sin\phi,\\
 u_{k-1}\sin\phi - u_{k}\cos\phi, -u_{k+1}\cos\phi - u_{k+2}\sin\phi, u_{k+1}\sin\phi - u_{k+2}\cos\phi,..., \\
  -u_{p-1}\cos\phi - u_{p}\sin\phi, u_{p-1}\sin\phi - u_{p}\cos\phi, u_{p+1},..., u_n),\\
 \phi \in [0,\pi]
\end{gathered}
\end{equation*}
defines isotopy between $\varepsilon_p|_{fiber}$ and
\begin{equation}\label{isotop2}
\begin{gathered}
(u_1,...,u_n ) \longrightarrow  (-u_1, u_2,..., u_n).
\end{gathered}
\end{equation}
Similarly, $\varepsilon_n|_{fiber}$ is isotopic to $(\ref{isotop2})$

\textbf{3)} If $k,n$ are odd and $p$ is even, then arguing as before we can prove that $\varepsilon_p$ is isotopic to identity and $\varepsilon_n$ is isotopic to
\begin{equation*}
\begin{gathered}
(u_1,...,u_n ) \longrightarrow \\
  (-u_1, u_2,...,u_{n-1}, -u_n).
\end{gathered}
\end{equation*}

\textbf{4)}Arguing in the same way we can prove that $\varepsilon_p$ is isotopic to
\begin{equation*}
\begin{gathered}
(u_1,u_2,...,u_n ) \longrightarrow   ((-1)^p u_1, u_2,...,u_n).
\end{gathered}
\end{equation*}
and $\varepsilon_n$ is isotopic to
\begin{equation*}
\begin{gathered}
(u_1,...,u_n ) \longrightarrow   ((-1)^ku_1, u_2,...,u_{n-1}, (-1)^{n-p}u_n).
\end{gathered}
\end{equation*}

Let us note that
\begin{equation*}
\begin{gathered}
v_1(u) = (u_1, ... ,u_k, u_{k+1}, ...,u_p,0, ...,0)\\
v_2(u) = (u_1, ... ,u_k,0, ...,0, u_{p+1}, ...,u_n)
\end{gathered}
\end{equation*}
are normal vectors to $\mathcal{R}$ at point $u$. Moreover $v_j(\varepsilon_iu) = \varepsilon_iv_j(u)$ for $i=p,n$ and $j=1,2$. Hence, the transition maps $\varepsilon_p, \varepsilon_n$ preserve the orientation of the normal bundle of $\mathcal{R}$. If maps (\ref{trans}) change the orientation of $\mathbb{R}^n$, then they change the orientation of the tangent bundle. So, we obtain that if $p$ is even, then $\varepsilon_p$ preserves the orientation. If $\varepsilon_n$ preserves the orientation, then numbers $k$ and $n-p$ are both simultaneously even or odd (or equivalently $n-p+k$ is even).
\end{proof}

We get that for a fixed $p, n$ the fibrations are isomorphic for all even $k$ (is trivial). Also, for a fixed $p, n$ fibrations are isomorphic for all odd $k$. In other words, if numbers $p$ and $n$ are fixed, then  $\mathcal{N}(k_1, p,n)$ is diffeomorphic to $\mathcal{N}(k_2,p,n)$ if and only if $k_1$, $k_2$ are both even or both are odd.

Denote by $L_k$ the embedded Lagrangian, i.e.
\begin{equation*}
L_k = \psi(\mathcal{N}(k,p,n)) \subset \mathbb{C}^n.
\end{equation*}

\begin{lemma}
The embedded Lagrangians $L_k$ are monotone with monotonicity constant $\frac{\pi}{2}$ and with minimal Maslov number $N_{L_k} = \gcd(p, n-p+k)$.
\end{lemma}

\begin{proof}

Let us note that the polytope $P_k$ is Delzant and Fano. From Theorem $\ref{Fanos}$ and Theorem $\ref{embdelzant}$ (or Lemma $\ref{l}$) we get that $L_k$ is monotone embedded Lagrangian. Unfortunately, Theorem $\ref{Fanos}$ doesn't help us to find the minimal Maslov number. Let us consider $1-$cycles
\begin{equation*}
\begin{gathered}
e_1(s) = \psi(u_1,...,u_n, 2s\varepsilon_p), \;\; e_2(s) = \psi(u_1,...,u_n, 2s\varepsilon_n), \\
s \in [0,1], \;\; u_1,..,u_n=const, \;\; \varepsilon_p = (1,0), \;\; \varepsilon_n = (0,1),
\end{gathered}
\end{equation*}
and cycles $r_1(s)$, $r_2(s)$ as in Lemma $\ref{class}$. Let $d_1,...d_m$ be elements defined in Lemma $\ref{areaform}$. If $p>2$ and $n-p>2$, then the fiber $S^{p-1} \times S^{n-p-1}$ is simply connected and $H_1(L_k, \mathbb{Z}) = H_1(T^2, \mathbb{R}) = \mathbb{R}^2$ and $e_1 = 2r_1$, $e_2 = 2r_2$ as elements of $H_1(L_k, \mathbb{Z})$.  Then from formula $(\ref{lagrmasl2})$  we obtain
\begin{equation}\label{exm1}
\begin{gathered}
\mu(r_1) = \frac{1}{2}\mu(e_1) = p \\
\mu(r_2) = \frac{1}{2}\mu(e_2) = n-p+k \\
\mu(d_j) = 0, \quad j=1,...,m
\end{gathered}
\end{equation}
We get that the minimal Maslov number $N_{L_k} = \gcd(p, n-p+k)$. From formula (\ref{area}) we have
\begin{equation}\label{exar1}
\begin{gathered}
\psi^{*}(\lambda)(r_1) = \frac{1}{2}\psi^{*}(\lambda)(e_1) = \frac{\pi}{2}p, \\
 \psi^{*}(\lambda)(r_2) = \frac{1}{2}\psi^{*}(\lambda)(e_2) = \frac{\pi}{2}(n-p+k) \\
\psi^{*}(\lambda)(d_j) = 0, \quad j=1,...,m
\end{gathered}
\end{equation}
The lemma follows from $(\ref{exm1})$ and $(\ref{exar1})$.
\end{proof}

Let us construct some explicit examples. Let us fix numbers $p, n$ and varying $k$ construct monotone Lagrangians distinct up to Lagrangian isotopy. Let us recall that if $k, p, n$ are even, then the diffeomorphism type of $L_k$ is independent of $k$ but the Maslov class depends on $k$.

Let $n$ be an even number. Consider monotone embeddings $L_{k} = \psi(\mathcal{N}(k, 4, n))$, where $n \geqslant 10$, $k=0, 2$. All parameters in these Lagrangians satisfy inequalities of $(\ref{eq})$, i.e. $n-4+k > 4$. \\
If $n = 0$ mod $4$,
\begin{equation*}
N_{L_0} = gcd(4, n-4) = 4, \quad N_{L_2} = gcd(4, n-2) = 2.
\end{equation*}
If $n = 2$ mod $4$, then
\begin{equation*}
N_{L_0} = gcd(4, n-4) = 2, \quad N_{L_2} = gcd(4, n-2) = 4.
\end{equation*}
In this example all parameters are even. Hence, the fibration is trivial. So, for each even $n \geqslant 10$  we get $2$ monotone embeddings of $S^{3}\times S^{n - 5} \times T^2$ into $\mathbb{C}^{n}$ as a monotone Lagrangian submanifolds with different minimal Maslov number. This implies that our embeddings are not Lagrangian isotopic.
\\

Let $n$ be an even number. Assume that $L_{k} = \psi(\mathcal{N}(k, 8, 2n ))$, where $n \geqslant 10$,  and $k=0,2,4,6$. We see that $2n-8+k>8$ and all inequalities are satisfied.\\
If $n=0$ mod $4$, then
\begin{equation*}
\begin{gathered}
N_{L_0} = gcd(8, 2n-8) = 8, \quad N_{L_2} = gcd(8, 2n-6) = 2, \\
N_{L_{4}} = gcd(8, 2n-4) = 4, \quad   N_{L_{6}} = gcd(8, 2n-2) = 2.
\end{gathered}
\end{equation*}
If $n=2$ mod $4$, then
\begin{equation*}
\begin{gathered}
N_{L_0} = gcd(8, 2n-8) = 4, \quad N_{L_2} = gcd(8, 2n-6) = 2, \\
N_{L_{4}} = gcd(8, 2n-4) = 8, \quad   N_{L_{6}} = gcd(8, 2n-2) = 2.
\end{gathered}
\end{equation*}
All parameters are even and the fibration is trivial. Therefore, we obtain that for each $n \geqslant 10$ we have at least $3$ monotone embeddings of $S^{7}\times S^{2n - 9} \times T^2$ into $\mathbb{C}^{2n}$ distinct up to Lagrangian isotopy with minimal Maslov numbers $2, 4, 8$.
\\

Let $n$ and $k$ be even numbers. Let us consider $L_{k} = \psi(\mathcal{N}(k, 24, 6n ))$, where $n \geqslant 10$, and $0 \leqslant k \leqslant 22$.\\
If $n=0$ mod $4$, then
\begin{equation*}
\begin{gathered}
N_{L_0} = gcd(24, 6n-24) = 24, \quad N_{L_2} = gcd(24, 6n-22) = 2,  \\
N_{L_{4}} = gcd(24, 6n-20) = 4,  \quad  N_{L_{6}} = gcd(24, 6n-18) = 6, \\
N_{L_{8}} = gcd(24, 6n-16) = 8 \quad N_{L_{12}} = gcd(24, 6n-12) = 12.
\end{gathered}
\end{equation*}
If $n=2$ mod $4$, then
\begin{equation*}
\begin{gathered}
N_{L_0} = gcd(24, 6n-24) = 12, \quad N_{L_2} = gcd(24, 6n-22) = 2,  \\
N_{L_{4}} = gcd(24, 6n-20) = 8,  \quad  N_{L_{6}} = gcd(24, 6n-18) = 6, \\
N_{L_{8}} = gcd(24, 6n-16) = 4 \quad N_{L_{12}} = gcd(24, 6n-12) = 24.
\end{gathered}
\end{equation*}
This means that we have at least 6 embeddings distinct up to Lagrangian isotopy. Theorem $\ref{Haef}$ says that we can not have more than $4$ embeddings of $S^{p-1}\times S^{n-p-1} \times T^2$ into $\mathbb{C}^n$ distinct up to smooth isotopy. Hence, at least two of our embeddings are smoothly isotopic but they are not Lagrangian isotopic. We proved the following lemma:

\begin{lemma}\label{lemmasp}
Let $n$ be an even number greater than $9$. There exist at least $6$ embeddings of $S^{23}\times S^{6n-25} \times T^2$ into $\mathbb{C}^{6n}$ distinct up to Lagrangian isotopy. At least two of these embeddings are smoothly isotopic but they are not Lagrangian isotopic.
\end{lemma}

Let us prove more general lemma.

\begin{lemma}\label{lemmasp2}
Let $m$ be an arbitrary integer greater than or equal to $3$. Let $n$ be an even number greater than $9$. There exist at least $2m$ embeddings of $S^{3\cdot2^{m}-1}\times S^{3\cdot2^{m-2}(n-4)-1} \times T^2$ into $\mathbb{C}^{3\cdot2^{m-2}n}$ distinct up to Lagrangian isotopy. At least $]\frac{m}{2}[ $ of these embeddings are smoothly isotopic but not Lagrangian isotopic, where $]\frac{m}{2}[$ is the smallest integer greater than or equal to $\frac{m}{2}$. If $m=3$, then we get Lemma $\ref{lemmasp}$.
\end{lemma}
\begin{proof}
Suppose that $k$ is an even number. We consider $L_k = \psi(\mathcal{N}(k, \; 3\cdot2^m, \; 3\cdot2^{m-2}n) = \psi(S^{3\cdot2^{m}-1}\times S^{3\cdot2^{m-2}(n-4)-1} \times T^2)$. We have $3\cdot2^{m-2}n - 3\cdot2^{m+1}>0$. Therefore, $k$ can be an arbitrary positive even number less than $3\cdot2^m$. \\
\\
If $n = 0$ mod $4$, then
\begin{equation*}
\begin{gathered}
N_{L_{0}} = gcd(3\cdot2^m, \; 3\cdot2^{m-2}(n-4)) = 3\cdot2^m,\\
N_{L_{2^l}} = gcd(3\cdot2^m, \; 3\cdot2^{m-2}(n-4) + 2^l) = 2^l, \quad l=1,...,m\\
N_{L_{3\cdot 2^l}} = gcd(3\cdot2^m, \; 3\cdot2^{m-2}(n-4) + 3\cdot 2^l) = 3\cdot 2^l, \quad l=1,...,m-1
\end{gathered}
\end{equation*}

\vspace*{.1in}

If $n = 2$ mod $4$, then
\begin{equation*}
\begin{gathered}
N_{L_{3\cdot2^{m-1}}} = gcd(3\cdot2^m, \; 3\cdot2^{m-2}(n-4) + 3\cdot2^{m-1} ) = 3\cdot2^{m-1}\\
N_{L_{3\cdot2^{m-1} - 2^l}} = gcd(3\cdot2^m, \; 3\cdot2^{m-2}(n-4) + 3\cdot2^{m-1} - 2^l) = 2^l, \;\; l=1,...,m\\
N_{L_{3\cdot2^{m-1} - 3\cdot 2^l}} = gcd(3\cdot2^m, \; 3\cdot2^{m-2}(n-4) + 3\cdot2^{m-1} - 3\cdot2^l) = 3\cdot 2^l, \;\; l=1,...,m-1
\end{gathered}
\end{equation*}
In both cases we get $2m$ embeddings distinct up to Lagrangian isotopy. From Theorem $\ref{Haef}$ we know that there can not be more that $4$ embeddings of $\mathcal{N}(k, \; 3\cdot2^m, \; 3\cdot2^{m-2}n)$ distinct up to smooth isotopy and this implies the last statement of our lemma.
\end{proof}

Let us consider  some examples of nonorientable submanifolds. Assume that $L_{2k+1} = \psi(\mathcal{N}(2k+1, 5, 5m+2))$, where $m \geqslant 2$ and is odd. then
\begin{equation*}
\begin{gathered}
N_{L_1} = 1, \quad N_{L_3} = 5.
\end{gathered}
\end{equation*}
We have $2$ monotone embeddings with different minimal Maslov number. Our manifold is a fibration over $T^2$ with fiber $S^{4}\times S^{5m-4}$.

Let  $L_{2k+1} = \psi(\mathcal{N}(2k+1, \; 2\cdot 3^5, \; 2\cdot 3^5m+2\cdot 3^5))$, where $m\geqslant 2$. Then
\begin{equation*}
\begin{gathered}
N_{L_1} = 1, \quad N_{L_3} = 3, \quad N_{L_9} = 9, \quad N_{L_{27}} = 27, \quad N_{L_{81}} = 81 \quad N_{L_{243}} = 243.
\end{gathered}
\end{equation*}
We get $6$ embeddings with different minimal Maslov number.

\section{Proof of Theorem $\ref{ex2}$}

Let $\mathcal{R}$ be defined by
\begin{equation}\label{ex4}
\begin{gathered}
\left\{
 \begin{array}{l}
\; \; \; \sum\limits_{m=1}^q u_m^2  + \sum\limits_{m=q+1}^l u_m^2 \quad \; \; + \quad \; \; \sum\limits_{m=p+1}^n u_m^2 = n-p+l \\
-\sum\limits_{m=1}^q u_m^2   \; \quad +  \; \;  \quad \sum\limits_{m=l+1}^k u_m^2  \quad \quad \quad = \quad \quad  k-l-q \\
\; \; \; \sum\limits_{m=1}^q u_m^2 \; \;   \quad \quad + \; \;\quad \quad \; \; \; \sum\limits_{m=k+1}^p u_m^2 \; \; \; = \; \; \; p-k+q
 \end{array}
\right.
\\  q < l < k < p < n, \quad k-l-q  < 0, \quad n-p + k - q < p-l
\end{gathered}
\end{equation}

\textbf{Remark.} Let us mention that the system above corresponds to product of three simplices. If $l$ is even, then the diffeomorphism type of $\mathcal{R}$ is independent of $l$. We will show below that the Maslov class depends on $l$. Hence,  varying $l$ we can construct embedded monotone Lagrangians distinct up to Lagrangian isotopy.
\\

The lattice $\Lambda$ is generated by
\begin{equation*}
\begin{gathered}
\gamma_1=...\gamma_q  = (1,-1,1)^T \quad \gamma_{q+1} = ... = \gamma_l =  \gamma_{p+1}=...=\gamma_n = (1,0,0)^T \\
\gamma_{l+1} = ... = \gamma_k = (0,1,0)^T  \quad \gamma_{k+1}=...=\gamma_p=(0,0,1)^T
\end{gathered}
\end{equation*}

Let us note that under our conditions $\Lambda_u = \Lambda$ for all $u \in \mathcal{R}$. Therefore, by Lemma $\ref{l}$ submanifold $\psi(\mathcal{N})$ is embedded.

Let us study the topology of $\mathcal{N}$. By $f_1, f_2, f_3$ denote the first, second, and the third equations of our system. Then system (\ref{ex4}) is equivalent to
\begin{equation*}
\begin{gathered}
\left\{
 \begin{array}{l}
f_1 + f_2 + f_3 = n \\
(p-k+q)f_2 - (k-l-q)f_3 = 0 \\
(p-k+q)(f_1 + f_2 +f_3) - nf_3 = 0
 \end{array}
\right.
\end{gathered}
\end{equation*}
and we have
\begin{equation*}
\begin{gathered}
\left\{
 \begin{array}{l}
u_1^2 + ... + u_n^2= n \\
\\
-(p-l)\sum\limits_{m=1}^q u_m^2 + (p-k+q)\sum\limits_{m=l+1}^k u_m^2 - (k-l-q)\sum\limits_{m=k+1}^p u_m^2 = 0 \\
 \\
-(n-p+k-q)\sum\limits_{m=1}^q u_m^2 + (p-k+q)\sum\limits_{m=q+1}^l u_m^2 + (p-k+q)\sum\limits_{m=l+1}^k u_m^2 -\\
 \quad - (n-p+k-q)\sum\limits_{m=k+1}^p u_m^2 +  (p-k+q)\sum\limits_{m=p+1}^n = 0
 \end{array}
\right.
\end{gathered}
\end{equation*}
We want to use Theorem $\ref{threequad}$  to study the topology of $\mathcal{R}$. Let us use the same notations as in Theorem  $\ref{threequad}$.
\begin{figure}[h]
\includegraphics[width=0.9\linewidth]{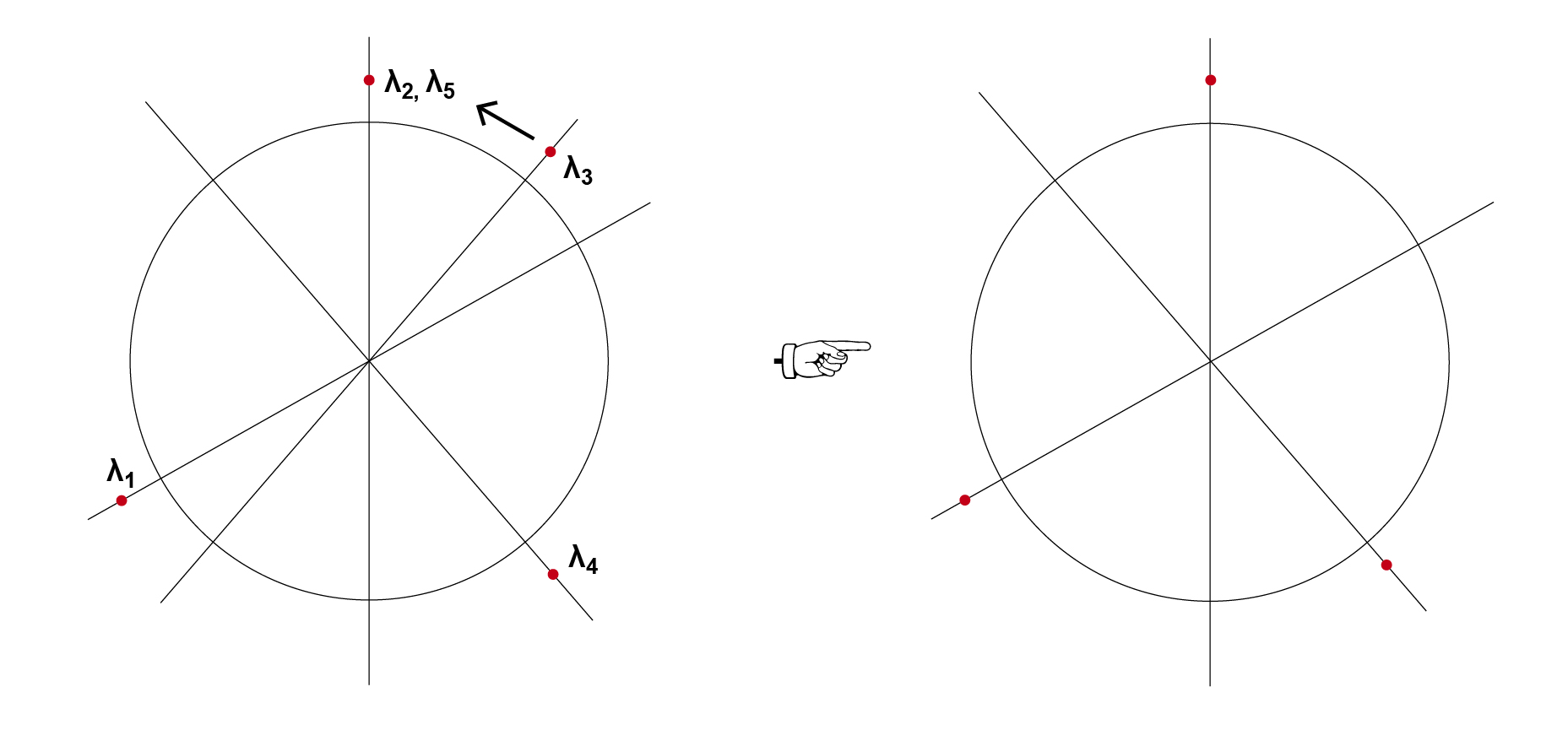}
\end{figure}

\begin{equation*}
\begin{gathered}
\lambda_1 = (-(p-l), -(n-p+k-q)) \quad n_1 = q, \\
\lambda_2 = (0, p-k+q) \quad n_2 = l-q, \\
\lambda_3 = (p-k+q, p-k+q) \quad n_3 = k-l, \\
\lambda_4 = (-(k-l-q), -(n-p+k-q))  \quad n_4 = p-k, \\
\lambda_5 = (0,p-k+q) \quad  n_5=n-p
\end{gathered}
\end{equation*}
We assumed that $p-l > n-p+k-q$ and $k-l-q<0$ Hence, without breaking the regularity condition ($0 \in \mathbb{R}^2$ doesn't belong to the line interval connecting any two of the $\lambda_i$) we can deform points $\lambda_2$, $\lambda_3$, $\lambda_5$  to $(0,1)$ and vectors $\lambda_4$, $\lambda_1$ (for $\lambda_1$ we use the inequality from $(\ref{ex4})$) to $(1,-1)$, $(-2,-1)$ respectively, where $(0,1)$ comes with multiplicity $n_2+n_3+n_5 = n-p+k-q$. Then Theorem $\ref{threequad}$ says that
\begin{equation*}
\mathcal{R} \cong S^{n-p+k -q-1} \times S^{p - k - 1} \times S^{q - 1}.
\end{equation*}

Let us study the topology of $\mathcal{N}$. From (\ref{ex4}) we see that the lattice $\Lambda^{*}$ is generated by $\varepsilon_l = (1,0,0)$, $\varepsilon_k = (0,1,0)$, $\varepsilon_p = (0,0,1)$, and $D_{\Gamma} = \mathbb{Z}_2^3$. We know that $\mathcal{R} \times_{D_{\Gamma}} T_{\Gamma} \rightarrow T_{\Gamma}/D_{\Gamma} = T^3$ is fibration with fiber $\mathcal{R}$. Let us denote $\mathcal{R} \times_{D_{\Gamma}} T_{\Gamma}$ by $\mathcal{N}(q,l,k,p,n)$.

\begin{lemma}
If $q,l,k,p,n$ are even numbers, then the fibration is trivial and $\mathcal{N}(q,l,k,p,n) = S^{n-p+k -q-1} \times S^{p - k - 1} \times S^{q - 1} \times T^3$. The fiber bundle is orientable if and only if numbers $n-p+l, \; \; k-l+q, \;\; p-k+q$ are even.
\end{lemma}

\begin{proof}
Involutions $\varepsilon_l, \varepsilon_k, \varepsilon_p$ act on $\mathcal{R}$ by
\begin{equation*}
\begin{gathered}
\varepsilon_l(u_1,...,u_q,u_{q+1},...,u_l,u_{l+1},...,u_k,u_{k+1},...,u_p, u_{p+1}, ..., u_n) = \\
  (-u_1,...,-u_q,-u_{q+1},...,-u_l,u_{l+1},...,u_k,u_{k+1},...,u_p, -u_{p+1}, ..., -u_n), \\
\varepsilon_k(u_1,...,u_q,u_{q+1},...,u_l,u_{l+1},...,u_k,u_{k+1},...,u_p, u_{p+1}, ..., u_n) = \\
  (-u_1,...,-u_q,u_{q+1},...,u_l,-u_{l+1},...,-u_k,u_{k+1},...,u_p, u_{p+1}, ..., u_n), \\
\varepsilon_p(u_1,...,u_q,u_{q+1},...,u_l,u_{l+1},...,u_k,u_{k+1},...,u_p, u_{p+1}, ..., u_n) = \\
  (-u_1,...,-u_q,u_{q+1},...,u_l,u_{l+1},...,u_k, -u_{k+1},..., -u_p, u_{p+1}, ..., u_n).
\end{gathered}
\end{equation*}
\\

The fibration is orientable if and only if involutions $\varepsilon_l, \varepsilon_k, \varepsilon_p$ preserve the orientation of $\mathcal{R}$. Arguing as in Lemma $\ref{orientation}$ we get that $\varepsilon_l$ is isotopic to
\begin{equation*}
\begin{gathered}
(u_1,.., u_n) \rightarrow ( (-1)^{l}u_1,u_2,...,u_{n-1}, (-1)^{n-p} u_n).
\end{gathered}
\end{equation*}
We see that $\varepsilon_l$ preserves the orientation if and only if either $l$, $n-p$ are even, or $l$, $n-p$ are odd (see the end of the proof of Lemma $\ref{orientation}$). In other words, $\varepsilon_l$ preserves the orientation if and only if $n-p+l$ is even. In the same way, $\varepsilon_k$ is isotopic to
\begin{equation*}
\begin{gathered}
(u_1,.., u_n) \rightarrow ((-1)^{q}u_1, u_2,...,u_{k-1},(-1)^{k-l}u_{k}, u_{k+1},...,u_n)
\end{gathered}
\end{equation*}
and preserves the orientation if and only if $q+k-l$ is even. Finally,  $\varepsilon_p$ is isotopic to
\begin{equation*}
\begin{gathered}
(u_1,.., u_n) \rightarrow ((-1)^{q}u_1, u_2,...,u_{p-1},(-1)^{p-k}u_{p}, u_{p+1},...,u_n)
\end{gathered}
\end{equation*}
and preserves the orientation if and only if $q+p-k$ is even.

We see that if $q,l,k,p,n$ are even, then all involutions are isotopic to identity and the fiber bundle is trivial.

\end{proof}

Let $L_l$ be our embedded Lagrangian, i.e.
\begin{equation*}
L_l = \psi(\mathcal{N}(q,l,k,p,n)) \subset \mathbb{C}^n.
\end{equation*}

\begin{lemma}
Lagrangian $L_l$ is  embedded and monotone with the minimal Maslov number $gcd(n-p+l, q+l-k, p-k+q)$.
\end{lemma}

\begin{proof}
We already mentioned in the beginning of the proof that $L_l$ is embedded because $\Lambda_u = \Lambda$ for all $u \in \mathcal{R}$ (see Theorem $\ref{l}$).

Let us consider cycles $e_i(s)$, $r_i(s)$ as in Lemma $\ref{class}$, where $i=1,2,3$. Let $d_1,...,d_m$ be elements defined in Lemma $\ref{areaform}$. From formulas $(\ref{lagrmasl2})$ and (\ref{area}) we can find the Maslov class and the symplectic area form
\begin{equation*}
\begin{gathered}
\mu(r_1) = n-p+l, \;\; \mu(r_2) = k-l-q, \;\; \mu(r_3)  = p-k+q\\
\psi^{*}(\lambda)(r_1)=  \frac{\pi}{2}(n-p+l), \;\; \psi^{*}(\lambda)(r_2)=  \frac{\pi}{2}(k-l-q), \psi^{*}(\lambda)(r_3)= \frac{\pi}{2}(p-k+q)\\
\mu(d_j) = \psi^{*}(\lambda)(d_j) = 0, \quad j=1,...,m
\end{gathered}
\end{equation*}
So, we get that $L_l$ is monotone embedded Lagrangian with monotonicity constant $\frac{\pi}{2}$ and with minimal Maslov number $gcd(n-p+l, \; l+q-k, \; p-k+q)$.

\end{proof}

Let us consider some examples. Let
\begin{equation*}
L_{10} = \psi(\mathcal{N}(8,10,16,24,26)) = \psi(S^9 \times S^7 \times S^7  \times T^3) \subset \mathbb{C}^{26}.
\end{equation*}
We have that $L_{10}$ is embedded monotone Lagrangian with minimal Maslov number $N_{L_{10}} = gcd(12,2,16) = 2$.

We can construct more examples

\begin{lemma}
For any $m \geqslant 5$, there exist at least $6$ monotone Lagrangian embeddings of $S^{11} \times S^{47} \times S^{24m-37}  \times T^3$ into $\mathbb{C}^{24m + 24}$ distinct up to Lagrangian isotopy.
\end{lemma}

\begin{proof}
Let us assume that
\begin{equation*}
q=12,\; k=36, \; p=24m, \; n=24m+24.
\end{equation*}
From  $(\ref{ex4})$ we have inequalities
\begin{equation*}
k-l-q < 0, \; \; \; n-p +k - q < p-l \;\; \;  \Leftrightarrow \;\; 24 < l < 24m -48.
\end{equation*}
Suppose $l$ is an even number. Then we have
\begin{equation*}
L_{l} = \psi(\mathcal{N}(12, l, 36, 24m, 24m + 24)) = \psi(S^{47} \times S^{24m-37} \times S^{11}  \times T^3) \subset \mathbb{C}^{24m + 24}.
\end{equation*}
Direct calculations show that we have Lagrangians with
\begin{equation*}
N_{L_{26}} = 2, \; \; \; \; N_{L_{28}} = 4, \; \; \; \; N_{L_{30}} = 6, \; \; \; \; N_{L_{32}} = 8, \; \; \; \; N_{L_{36}} = 12, \; \; \; \; N_{L_{48}} = 24.
\end{equation*}
These Lagrangians have different minimal Maslov numbers and are not Lagrangian isotopic.
\end{proof}

We also can construct smoothly isotopic submanifolds which are not Lagrangian isotopic.
\begin{lemma}
For any $m \geqslant 5$, there exist at least $8$ monotone Lagrangian embeddings of $S^{23} \times S^{191} \times S^{96m-121}  \times T^3$ into $\mathbb{C}^{96m - 121}$ distinct up to Lagrangian isotopy. At least two of these embeddings are smoothly isotopic.
\end{lemma}
Assume that
\begin{equation*}
q=24,\; k=120, \; p=96m, \; n=96m + 96.
\end{equation*}
From  $(\ref{ex4})$ we have inequalities
\begin{equation*}
k-l-q < 0, \; \; \; n - p + k - q < p - l \;\;  \Leftrightarrow  \; \; 96 < l < 96m - 192.
\end{equation*}
Suppose $l$ is an even number. We have
\begin{equation*}
L_{l} = \psi(\mathcal{N}(24, l, 120, 96m,, 96m + 96)) = \psi(S^{191} \times S^{96m-121} \times S^{23}  \times T^3) \subset \mathbb{C}^{96m + 96},
\end{equation*}
Then minimal Maslov number is given by
\begin{equation*}
\begin{gathered}
N_{L_{98}} = 2, \; \;  N_{L_{100}} = 4, \; \;  N_{L_{102}} = 6, \; \;  N_{L_{104}} = 8, \; \; N_{L_{108}} = 12 \; \; N_{L_{112}} = 16, \\
N_{L_{120}} = 24, \; \; N_{L_{128}} = 32, \; \;  N_{L_{144}} = 48, \; \;  N_{L_{192}} = 96.
\end{gathered}
\end{equation*}
Theorem $\ref{Haef}$ says that we can not have more than $8$ smoothly not isotopic embeddings of $\mathcal{N}(24, 2l, 120, 96m,, 96m + 96)$. We constructed $10$ embeddings with different Maslov number. This means that we have at least two smoothly isotopic embeddings that are not Lagrangian isotopic.
\\

\section{Proof of Theorem $\ref{th3}$}

Let us consider a pentagon $P$ in $\mathbb{R}^2$ defined by inequalities
\begin{equation*}
\left\{
 \begin{array}{l}
x_1 + 1\geqslant 0, \quad   x_2 +1 \geqslant 0  \\
-x_1 + 1 \geqslant 0, \quad   -x_2 + 1 \geqslant 0, \\
-x_1 - x_2 + 1  \geqslant 0
 \end{array}
\right.
\end{equation*}
We have $a_1=(1,0)$, $a_2 = (0,1)$, $a_3 = (-1,0)$, $a_4 = (0,-1)$, $a_5=(-1,-1)$, $b=(1,1,1,1,1)^T$. So, the system of quadrics associated to $P$ has the form
\begin{equation}\label{pentagon3}
\left\{
 \begin{array}{l}
u_1^2 + u_3^2  = 2 \\
u_2^2 + u_4^2 = 2 \\
u_1^2 + u_2^2 + u_5^2 = 3
 \end{array}
\right.
\end{equation}
and $\gamma_1 = (1,0,1)^T$, $\gamma_2=(0,1,1)^T$, $\gamma_3 = (1,0,0)$, $\gamma_4 = (0,1,0)^T$, $\gamma_5 = (0,0,1)^T$. Our pentagon is obtained from a square by a vertex truncation. From Lemma $\ref{lemsurface}$ or Theorem $\ref{threequad}$ we have that the manifold $\mathcal{R}$ is diffeomorphic to an oriented surface of genus $5$.  Let $S_5$ be our surface.

We have that the lattice $\Lambda$ is generated by
\begin{equation*}
\gamma_3 = (1,0,0)^T, \quad \gamma_4 = (0,1,0)^T, \quad \gamma_5 = (0,0,1)^T.
\end{equation*}
And dual lattice $\Lambda^{*}$ is generated by
\begin{equation*}
\varepsilon_3 = (1,0,0), \quad \varepsilon_4 = (0,1,0), \quad \varepsilon_5 = (0,0,1).
\end{equation*}
The group $D_{\Gamma}$ acts on $S_5$ by
\begin{equation*}
\begin{gathered}
\varepsilon_3(u_1,u_2,u_3,u_4,u_5) = (-u_1, u_2, -u_3, u_4, u_5), \\
\varepsilon_4(u_1,u_2,u_3,u_4,u_5) = (u_1, -u_2, u_3, -u_4, u_5), \\
\varepsilon_5(u_1,u_2,u_3,u_4,u_5) = (-u_1, -u_2, u_3, u_4, -u_5).
\end{gathered}
\end{equation*}
We obtain a fiber bundle
\begin{equation*}
\mathcal{N} = S_5 \times_{D_{\Gamma}} T_{\Gamma} \longrightarrow T^3
\end{equation*}
with fiber $S_5$. Arguing as in Lemma $\ref{orientation}$ we see that $\varepsilon_5$ doesn't preserve the orientation of $S_5$. Therefore, the fiber bundle is not orientable.

Our pentagon $P$ is Delzant and Fano. Therefore, from Theorem $\ref{Fanos}$ we have that $\mathcal{N}$ is embedded monotone Lagrangian.

Let us find the minimal Maslov number. Consider cycles $e_i(s)$, $r_i(s)$ as in Lemma $\ref{class}$, where $i=1,2,3$. Let $d_1,...,d_m$,be elements defined in Lemma $\ref{areaform}$. From formula $(\ref{lagrmasl2})$  we have
\begin{equation*}
\begin{gathered}
\mu(r_1) = \frac{1}{2}\mu(e_1) = 2, \quad \mu(r_2) = \frac{1}{2}\mu(e_2) = 2, \quad \mu(r_3) = \frac{1}{2}\mu(e_3) = 3,\\
\mu(d_j) = 0, \quad j=1,...,m.
\end{gathered}
\end{equation*}
We see that the minimal Maslov number is equal to $1$.

\section{Proof of Theorem $\ref{th4}$}

Now let us start with equation $(\ref{pentagon3})$ and increase the dimension of coordinate spaces, i.e. take a system

\begin{equation*}
\left\{
 \begin{array}{l}
\sum\limits_{m=1}^p u_m^2  \quad + \quad \quad  \sum\limits_{m=2p+1}^{3p}u_m^2  = 2p \\
\quad \quad \quad \quad \sum\limits_{m=p+1}^{2p}u_m^2 \quad   +  \quad \sum\limits_{m=3p+1}^{4p}u_m^2 = 2p \\
\sum\limits_{m=1}^pu_m^2 +  \sum\limits_{m=p+1}^{2p}u_m^2 \quad + \quad \quad \quad \sum\limits_{m=4p+1}^{5p}u_m^2 = 3p
 \end{array}
\right.
\end{equation*}

Let $P$ be the polytope corresponding to the system above. In toric topology, there exists an algorithm for finding $P$ (see \cite{Lopez4}). As in all previous theorems we need some parameter to vary the minimal Maslov number. So, let us develop the ideas of Theorem $\ref{th3}$ and consider the following system

\begin{equation}\label{ex7}
\begin{gathered}
\left\{
 \begin{array}{l}
\sum\limits_{m=1}^pu_m^2 \hspace{.4in} +  \hspace{.4in} \sum\limits_{m=2p+1}^{2p+q}u_m^2 + \sum\limits_{m=2p+q+1}^{3p}u_m^2   =   2p \\
\hspace{.65in}  \sum\limits_{m=p+1}^{2p}u_m^2  + \sum\limits_{m=2p+1}^{2p+q}u_m^2 \hspace{.25in} + \hspace{.25in} \sum\limits_{m=3p+1}^{4p}u_m^2  =   2p + q \\
\sum\limits_{m=1}^pu_m^2 +\sum\limits_{m=p+1}^{2p}u_m^2 \hspace{.85in} + \hspace{.85in}   \sum\limits_{m=4p+1}^{5p}u_m^2  =   3p
 \end{array}
\right. \\
q \leqslant p-1, \quad q > 0, \quad p \geqslant 2
\end{gathered}
\end{equation}

In the system above we added an additional parameter $q$. We will show below that the Maslov class of the Lagrangian associated to the system above depends on $q$. By simple manipulations and changes of coordinates we get

\begin{equation*}
\left\{
 \begin{array}{l}
\sum\limits_{m=1}^{5p}u_m^2 = 7p + q\\
\\
(2p+q)\sum\limits_{m=1}^pu_m^2 -  (p-q)\sum\limits_{m=p+1}^{2p}u_m^2 - 3p\sum\limits_{m=2p+1}^{2p+q}u_m^2 -\\
 - 6p\sum\limits_{m=3p+1}^{4p}u_m^2 + 2(2p+q)\sum\limits_{m=4p+1}^{5p}u_m^2 = 0 \\
 \\
(p+q)\sum\limits_{m=1}^pu_m^2 + (p+q)\sum\limits_{m=p+1}^{2p}u_m^2 - 6p\sum\limits_{m=2p+1}^{2p+q}u_m^2 - \\
- 6p\sum\limits_{m=2p+q+1}^{3p}u_m^2 - 6p\sum\limits_{m=3p+1}^{4p}u_m^2 + 2(4p+q)\sum\limits_{m=4p+1}^{5p}u_m^2= 0
 \end{array}
\right.
\end{equation*}

We use Theorem $\ref{threequad}$ to study the topology of system $(\ref{ex7})$. Denote by $\mathcal{R}$ the manifold associated to system $(\ref{ex7})$ We have
\begin{equation*}
\begin{gathered}
\lambda_1 = (2p+q, p+q) \quad n_1 = p, \\
\lambda_2 = (-(p-q), p+q) \quad n_2 = p, \\
\lambda_3 = (-3p, -6p) \quad n_3 = q, \\
\lambda_4 = (0, -6p)  \quad n_4 = p-q, \\
\lambda_5 = (-6p, -6p) \quad  n_5=p, \\
\lambda_6 = (2(2p + q), 2(4p + q))  \quad n_6 = p.
\end{gathered}
\end{equation*}
\begin{figure}[h]
\includegraphics[width=0.5\linewidth]{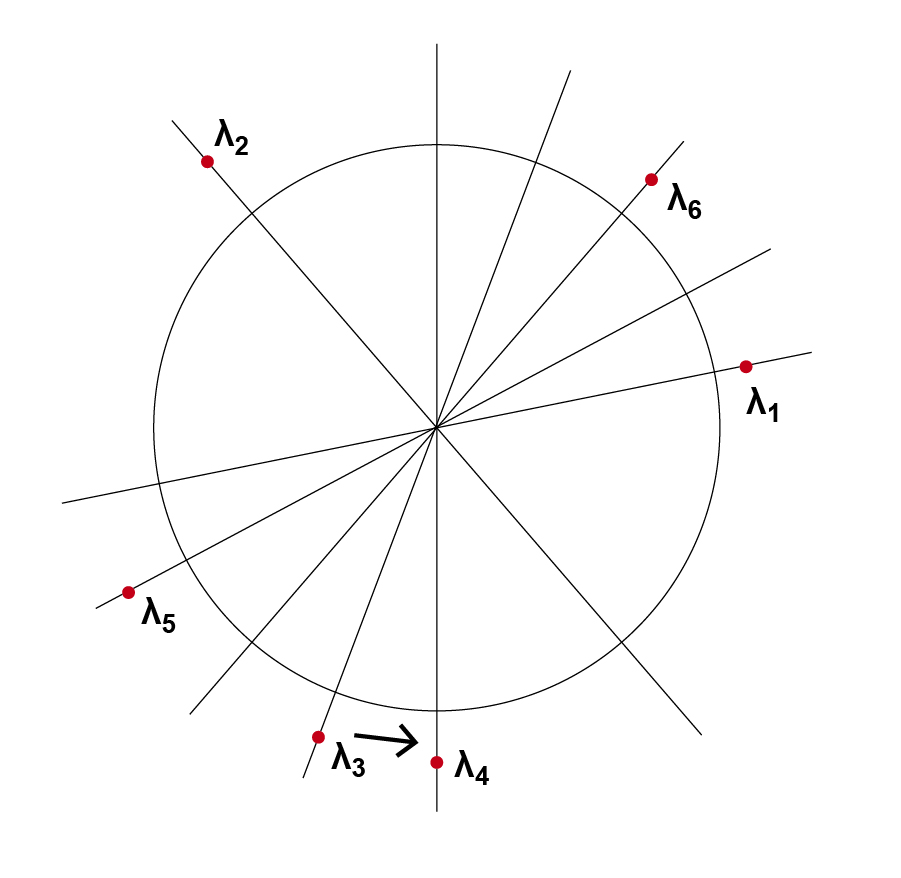}
\end{figure}
We can join points $\lambda_3$, $\lambda_4$ without breaking the regularity condition and get a point of multiplicity $n_3 + n_4 = p$. Then from Theorem $\ref{threequad}$ we see that
\begin{equation*}
\mathcal{R} = \#_5(S^{2p-1} \times S^{3p-2}).
\end{equation*}
Put
\begin{equation*}
\begin{gathered}
\mathcal{N}(q,p) = \mathcal{R} \times_{D_{\Gamma}} T^3, \quad p\geqslant 2, \\
\end{gathered}
\end{equation*}
We get that $\mathcal{N}(q,p)$ fibers over $T^3$ with fibre $\mathcal{R}$.

\begin{lemma}
The fibration is trivial if and only if numbers $p$ and $q$ are even. If the fiber bundle is orientable, then it is trivial.
\end{lemma}
\begin{proof}
From $(\ref{ex7})$ we have
\begin{equation*}
\begin{gathered}
\gamma_1=...=\gamma_p = (1,0,1)^T, \;\; \gamma_{p+1}=...=\gamma_{2p} = (0,1,1)^T, \\
\gamma_{2p+1}=...=\gamma_{2p+q} = (1,1,0)^T, \;\; \gamma_{2p+q+1}=...=\gamma_{3p} = (1,0,0)^T, \\
\gamma_{3p+1}=...=\gamma_{4p} = (0,1,0)^T, \;\; \gamma_{4p+1}=...=\gamma_{5p} = (0,0,1)^T.
\end{gathered}
\end{equation*}
We see that the lattice $\Lambda$ is isomorphic to $\mathbb{Z}^3$ and is generated by vectors $\gamma_{3p}$, $\gamma_{4p}$, $\gamma_{5p}$. The dual  lattice $\Lambda^{*}$ is generated by
$\varepsilon_{3p} = (1,0,0)$, $\varepsilon_{4p} = (0,1,0)$, $\varepsilon_{5p} = (0,0,1)$. Then the group $D_{\Gamma} \approx \mathbb{Z}_2^3$ acts on $\mathcal{R}$ by
\begin{equation*}
\begin{gathered}
\varepsilon_{3p}(u_1,...,u_{5p}) = \\
(-u_1,...,-u_p, u_{p+1},...,u_{2p},-u_{2p+1},...,-u_{3p}, u_{3p+1},..., u_{5p}), \\
\varepsilon_{4p}(u_1,...,u_{5p}) = \\
(u_1,...,u_p, -u_{p+1},...,-u_{2p+q},u_{2p+q+1},...,u_{3p}, -u_{3p+1},..., -u_{4p}, u_{4p+1},...,u_{5p}), \\
\varepsilon_{5p}(u_1,...,u_{5p}) =\\
 (-u_1,...,-u_{2p}, u_{p+1},..., u_{4p}, -u_{4p+1},...,-u_{5p}).
\end{gathered}
\end{equation*}
Suppose that numbers $p$, $q$ are even numbers. Then arguing as in Lemma $\ref{orientation}$ we see that transition maps $\varepsilon_{3p}, \varepsilon_{4p}, \varepsilon_{5p}$ are isotopic to
\begin{equation*}
\begin{gathered}
(u_1,...,u_{5p}) \rightarrow \\
 ((-1)^p u_1, u_2,...,u_{2p}, (-1)^qu_{2p+1}, u_{2p+2}, ..., u_{3p-1}, (-1)^{p-q}u_{3p}, u_{3p+1},..., u_{5p}), \\
(u_1,...,u_{5p}) \rightarrow \\
 (u_1,...,u_p, (-1)^{p}u_{p+1}, u_{p+2},...,u_{2p}, (-1)^qu_{2p+1}, u_{2p+2},...,u_{4p-1},(-1)^pu_{4p},u_{4p+1},...,u_{5p}), \\
(u_1,...,u_{5p}) \rightarrow \\
 ((-1)^p u_1, u_{2},...,u_p, (-1)^{p}u_{p+1}, u_{p+2},...,u_{5p-1}, (-1)^pu_{5p})
\end{gathered}
\end{equation*}
respectively. We see that $\varepsilon_{5p}$ preserves the orientation if and only if $p$ is even. If $p$ is even, then $\varepsilon_{4p}$ preserves the orientation if and only if $q$ is even. As a result we get that the fibration is orientable if and only $p$ and $q$ are even numbers. If both numbers are even, then all transition maps are isotopic to identity and the fiber bunde is trivial.

\end{proof}

So, we have that if $p, q$ are even numbers, then
\begin{equation*}
\mathcal{N}(q,p) = \#_5(S^{2p-1} \times S^{3p-2}) \times T^3, \;\; if \;\; p , q \;\; are \;\; even.
\end{equation*}
We see that the topology of $\mathcal{N}(q,p)$ is independent of $q$. Put
\begin{equation*}
L_q = \psi(\mathcal{N}(q,p)) \subset \mathbb{C}^n.
\end{equation*}
Note that $\Lambda_u = \Lambda$ for all $u \in \mathcal{R}$. Hence, by Lemma $\ref{l}$ the submanifold $L_q$ is embedded. Let us find the Maslov class and symplectic area homomorphism of $L_q$. Let us consider cycles $e_i(s)$, $r_i(s)$ as in Lemma $\ref{class}$, where $i=1,2,3$. We assumed that $p \geqslant 2$. Therefore, the fiber $\mathcal{R}$ is simply connected and $H_1(L_k, \mathbb{Z}) = H_1(T^3, \mathbb{Z}) = \mathbb{Z}^3$. From formulas $(\ref{lagrmasl2})$ and (\ref{area}) we get
\begin{equation*}
\begin{gathered}
\mu(r_1) = \frac{1}{2}\mu(e_1)= 2p, \quad \mu(r_2) = \frac{1}{2}\mu(e_2) = 2p + q, \quad \mu(r_3) = \frac{1}{2}\mu(e_3) = 3p,\\
\psi^{*}(\lambda)(r_1) = \frac{1}{2}\psi^{*}(\lambda)(e_1) = p\pi, \quad \psi^{*}(\lambda)(r_2) = \frac{1}{2}\psi^{*}(\lambda)(e_2) = \frac{(2p+q)\pi}{2},\\
\psi^{*}(\lambda)(r_3) = \frac{1}{2}\psi^{*}(\lambda)(e_3) = \frac{3p\pi}{2}.\\
\end{gathered}
\end{equation*}
We see that $L_q$ is embedded monotone Lagrangian with minimal Maslov number $gcd(2p, 2p+q, 3p) = gcd(p, q)$.
\\

Suppose that $L_{2q} = \psi(\mathcal{N}(2q, 12p)) =  \psi(\#_5 (S^{24p-1} \times S^{36p-2}) \times T^3) \subset \mathbb{C}^{120p}$. Then,
\begin{equation*}
\quad N_{L_2} = 2, \quad N_{L_4} = 4, \quad N_{L_6} = 6, \quad N_{L_{12}} = 12.
\end{equation*}
So, we have  $4$ embeddings with different Maslov number.

\begin{lemma}
Let $p$ be an arbitrary positive integer. There exist at least $10$ monotone Lagrangian embeddings of $\#_5 (S^{192p-1} \times S^{288p-2}) \times T^3$  into $\mathbb{C}^{480p}$ distinct up to Lagrangian isotopy. At least two of these embeddings are smoothly isotopic but they are not Lagrangian isotopic.
\end{lemma}

\begin{proof}
Assume that $L_{2q} = \psi(\mathcal{N}(2q, 96p)) =  \psi (5\# (S^{192p-1} \times S^{288p-2}) \times T^3 ) \subset \mathbb{C}^{480p}$. We get
\begin{equation*}
\begin{gathered}
 \quad N_{L_2} = 2, \quad N_{L_4} = 4, \quad N_{L_6} = 6, \quad N_{L_{8}} = 8, \quad N_{L_{12}}=12\\
\quad N_{L_{16}} = 16, \quad  N_{L_{24}} = 24, \quad  N_{L_{32}} = 32, \quad  N_{L_{48}} = 48, \quad  N_{L_{96}} = 96.
\end{gathered}
\end{equation*}
So, we have $10$ monotone Lagrangian embeddings distinct up to Lagrangian isotopy. From Theorem $\ref{Haef}$ we see that we can not have more than $8$ embeddings of $\mathcal{N}(2q, 96p)$ distinct up to smooth isotopy.  As a result we obtain smoothly isotopic submanifolds which are not Lagrangian isotopic.
\end{proof}

\section{Proof of Theorem $\ref{th5}$}

Let $P$  be a 6-gon defined by
\begin{equation*}
\left\{
 \begin{array}{l}
x_1 + 1\geqslant 0, \quad \; x_2 + 1 \geqslant 0  \\
-x_1 + 1 \geqslant 0, \quad \; -x_2 + 1 \geqslant 0, \\
-x_1 - x_2 + 1 \geqslant 0, \quad \; x_1 + x_2   +1 \geqslant 0
 \end{array}
\right.
\end{equation*}
We have $a_1=(1,0)$, $a_2 = (0,1)$, $a_3 = (-1,0)$, $a_4 = (0,-1)$, $a_5=(-1,-1)$, $a_6 = (1,1)$ $b=(1,1,1,1,1,1)^T$. The system of quadrics associated to $P$ has the form
\begin{equation*}
\left\{
 \begin{array}{l}
u_1^2 + u_3^2  = 2 \\
u_2^2 + u_4^2 = 2 \\
u_1^2 + u_2^2 + u_5^2 = 3 \\
u_1^2 + u_2^2 - u_6^2 = 1
 \end{array}
\right.
\end{equation*}
 and $\gamma_1 = (1,0,1,1)^T$, $\gamma_2 = (0,1,1,1)^T$, $\gamma_3 = (1,0,0,0)^T$, $\gamma_4 = (0,1,0,0)^T$, $\gamma_5 = (0,0,1,0)^T$, $\gamma_6 = (0,0,0,-1)^T$,

The 6-gon is obtained from a 5-gon by a vertex truncation. Using Lemma $\ref{lemsurface}$ or Theorem $\ref{Lop}$ we get that $\mathcal{R}$ is an orientable surface of genus 17.  Denote this surface by $S_{17}$.

The lattice $\Lambda$ is generated by
\begin{equation*}
\gamma_3 = (1,0,0,0), \quad \gamma_4 = (0,1,0,0), \quad \gamma_5 = (0,0,1,0), \quad \gamma_6 = (0,0,0,1).
\end{equation*}
And dual the lattice $\Lambda^{*}$ is generated by
\begin{equation*}
\varepsilon_3 = (1,0,0,0), \quad \varepsilon_4 = (0,1,0,0), \quad \varepsilon_5 = (0,0,1,0), \quad \varepsilon_6 = (0,0,0,1).
\end{equation*}
The group $D_{\Gamma} \approx \mathbb{Z}_2^4$ acts on $S_{17}$ by
\begin{equation*}
\begin{gathered}
\varepsilon_3(u_1,u_2,u_3,u_4,u_5) = (-u_1, u_2, -u_3, u_4, u_5, u_6), \\
\varepsilon_4(u_1,u_2,u_3,u_4,u_5) = (u_1, -u_2, u_3, -u_4, u_5, u_6), \\
\varepsilon_5(u_1,u_2,u_3,u_4,u_5) = (-u_1, -u_2, u_3, u_4, -u_5, u_6),\\
\varepsilon_6(u_1,u_2,u_3,u_4,u_5) = (-u_1, -u_2, u_3, u_4, u_5, -u_6).
\end{gathered}
\end{equation*}
We get a fiber bundle
\begin{equation*}
\mathcal{N} = S_{17} \times_{D_{\Gamma}} T_{\Gamma} \longrightarrow T^4
\end{equation*}
with fibre $S_{17}$. Arguing as in Lemma $\ref{orientation}$ we see that $\varepsilon_5$ and $\varepsilon_6$ don't preserve the orientation of $S_{17}$. Hence, the fiber bundle is not orientable.

The constructed 6-gon is Delzant and Fano, therefore $\psi(\mathcal{N})$ is monotone Lagrangian embedded  into $\mathbb{C}^6$.

Let us find the Maslov class. Let $e_i, r_i$ be elements as in Lemma $\ref{class}$. Let $d_1,...,d_m$ be elements defined in Lemma $\ref{areaform}$. From formula $(\ref{lagrmasl2})$  we have
\begin{equation*}
\begin{gathered}
\mu(r_1) = \frac{1}{2}\mu(e_1) = 2, \quad \mu(r_2) = \frac{1}{2}\mu(e_2) = 2, \quad \mu(r_3) = \frac{1}{2}\mu(e_3) = 3, \\
 \mu(r_4) = \frac{1}{2}\mu(e_4) = 1, \quad \mu(d_j) = 0, \quad j=1,...,m.
\end{gathered}
\end{equation*}
We see that the minimal Maslov number is equal to $1$.

\section{Proof of Theorem $\ref{th6}$}

Assume that $P$ is a pentagon defined by

\begin{equation*}
\left\{
 \begin{array}{l}
x_1 \geqslant 0, \quad x_2 \geqslant 0  \\
-(k - 1)x_1 + k \geqslant 0,\\
-(2k - 1)x_2 + 2k \geqslant 0, \\
-(k - 3)x_1 - (k + 2)x_2 + 2k \geqslant 0 \\
 \end{array}
\right. k \geqslant 4, \; \; k \in \mathbb{Z}
\end{equation*}
The associated system of quadrics has the form
\begin{equation*}
\left\{
 \begin{array}{l}
(k - 1)u_1^2 + u_3^2  = k \\
(2k - 1)u_2^2 + u_4^2 = 2k \\
(k - 3)u_1^2 + (k + 2)u_2^2 + u_5^2 = 2k \\
 \end{array}
\right. k \geqslant 4, \;\; k \in \mathbb{Z}
\end{equation*}
From Lemma $\ref{lemsurface}$ or Theorem $\ref{Lop}$ we have that the system of quadrics associated to any pentagon defines an oriented surface of genus $5$. Denote the surface by $S_5$. In this example we get a fiber bundle $\mathcal{N} \longrightarrow T^3$, where the fibre is $S_5$.

We have that $\gamma_1 = (k - 1,0, k - 3)$, $\gamma_2 = (0, 2k-1, k +2)$, $\gamma_3 = (1,0,0)$, $\gamma_4 = (0,1,0)$, $\gamma_5 = (0,0,1)$. The dual lattice $\Lambda^{*}$ is generated by $\varepsilon_3=(1,0,0)$, $\varepsilon_4 = (0,1,0)$,  $\varepsilon_5 = (0,0,1)$ and these elements act on $S_5$ by
\begin{equation*}
\begin{gathered}
\varepsilon_3(u_1,u_2,u_3,u_4,u_5) = ((-1)^{k - 1}u_1, u_2, -u_3, u_4, u_5)\\
\varepsilon_4(u_1,u_2,u_3,u_4,u_5) = (u_1, (-1)^{2k - 1}u_2, u_3, -u_4, u_5)\\
\varepsilon_5(u_1,u_2,u_3,u_4,u_5) = ((-1)^{k-3}u_1, (-1)^{k +2}u_2, u_3, u_4, -u_5)
\end{gathered}
\end{equation*}

Arguing as in Lemma $\ref{orientation}$ we see that if $k$ is even, then the fiber bundle is orientable. If $k$ is odd, then the fiber bundle is not orientable. Moreover, we see that if $k$ is even, then the diffeomorphism type of $\mathcal{N}$ is independent of $k$. Denote
\begin{equation*}
L = \psi(\mathcal{N})
\end{equation*}

Let us note that our pentagon is not Delzant. Hence, $\psi$ doesn't define an embedding but defines an immersion. As in Theorem $\ref{th3}$ we can find the Maslov class and symplectic are from. Let $r_1, r_2, r_3$ be elements of $H_1(L, \mathbb{R})$ as in Lemma $\ref{class}$. Let $d_1,...,d_m$ be elements defined in Lemma $\ref{areaform}$. From formulas $\ref{lagrmasl2}$ and $\ref{area}$ we get
\begin{equation*}
\begin{gathered}
\mu(r_1) = k, \quad \mu(r_2) = 2k, \quad \mu(r_3)  = 2k,\\
\psi^{*}(\lambda)(r_1) = \frac{\pi k}{2}, \quad \psi^{*}(\lambda)(r_2)  = \pi k, \quad \psi^{*}(\lambda)(r_3) = \pi k, \\
\mu(d_j) = \psi^{*}(\lambda) = 0, \quad j=1,,,m
\end{gathered}
\end{equation*}

We have that $L = \psi(\mathcal{N}) \subset \mathbb{C}^5$ is an immersed monotone Lagrangian and
\begin{equation*}
N_L = k.
\end{equation*}
Suppose that $k$ is even. Then the diffeomorphism type of $\mathcal{N}$ is independent of $k$, but the Maslov class of $L$ depends on $k$. As a result, we get infinitely many monotone immersions  of $\mathcal{N}$ into $\mathbb{C}^5$ distinct up to Lagrangian isotopy.

Department of Mathematics, Stony Brook University, Stony Brook, NY, 11794 USA\\
L.D. Landau Institute for Theoretical Physics, Chernogolovka,  142432, \\ Russia \\
Email address: vardan8oganesyan@gmail.com

\end{document}